\documentclass[11pt, notitlepage]{article}
\usepackage{amssymb,amsmath,comment}
\catcode`\@=11 \@addtoreset{equation}{section}
\def\thesection{\arabic{section}}

\def\theequation{\thesection.\arabic{equation}}
\catcode`\@=12
\usepackage{colortbl}%
\usepackage{a4wide}

\newcommand{\ds} {\displaystyle}
\newcommand{\e}{\epsilon}

\newcommand{\Om} {\Omega}
\newcommand{\ra} {\rightarrow}

\newcommand{\De} {\Delta}
\newcommand{\la} {\lambda}
\newcommand{\La} {\Lambda}
\newcommand{\noi} {\noindent}

\newcommand{\mb} {\mathbb}
\newcommand{\mc} {\mathcal}

\usepackage[all]{xy}
\catcode`\@=11
\def\theequation{\@arabic{\c@section}.\@arabic{\c@equation}}
\catcode`\@=12

\def\QED{\hfill {$\square$}\goodbreak \medskip}

\newtheorem{Theorem}{Theorem}[section]
\newtheorem{Lemma}[Theorem]{Lemma}
\newtheorem{Proposition}[Theorem]{Proposition}
\newtheorem{Corollary}[Theorem]{Corollary}
\newtheorem{Remark}[Theorem]{Remark}
\newtheorem{Definition}[Theorem]{Definition}

\begin{document}
{\vspace{0.01in}}

\title
{ \sc Positive solutions for nonlinear Choquard equation with
singular nonlinearity}

\author{
{\bf  Tuhina Mukherjee\footnote{email: tulimukh@gmail.com}}\; and\; {\bf K. Sreenadh\footnote{e-mail: sreenadh@gmail.com}}\\
{\small Department of Mathematics}, \\{\small Indian Institute of Technology Delhi}\\
{\small Hauz Khaz}, {\small New Delhi-16, India}\\
 }

\date{}

\maketitle

\begin{abstract}

\noi In this article, we study the following nonlinear Choquard equation with singular nonlinearity
\begin{equation*}
 \quad -\De u = \la u^{-q} + \left( \int_{\Om}\frac{|u|^{2^*_{\mu}}}{|x-y|^{\mu}}\mathrm{d}y \right)|u|^{2^*_{\mu}-2}u, \quad u>0 \; \text{in}\;
\Om,\quad u = 0 \; \mbox{on}\; \partial\Om,
\end{equation*}
where  $\Om$ is a bounded domain in $\mb{R}^n$ with smooth boundary $\partial \Om$, $n > 2,\; \la >0,\; 0 < q < 1, \; 0<\mu<n$ and $2^*_\mu=\frac{2n-\mu}{n-2}$. Using variational approach and structure of associated Nehari manifold, we show the existence and multiplicity of positive weak solutions of the above problem, if $\la$ is less than some positive constant.  We also study the regularity of these weak solutions.
\medskip

\noi \textbf{Key words:} Choquard equation, Nehari manifold, Singular nonlinearity.

\medskip

\noi \textit{2010 Mathematics Subject Classification:} 35R11, 35R09, 35A15.

\end{abstract}

\section{Introduction}
Let $\Om \subset \mb R^n$, $n>2$ be a bounded domain with smooth boundary $\partial \Om$.  We consider the following problem with singular nonlinearity :
\begin{equation*}
(P_{\la}): \quad
 \quad -\De u = \la u^{-q} + \left( \int_{\Om}\frac{|u|^{2^*_{\mu}}}{|x-y|^{\mu}}\mathrm{d}y \right)|u|^{2^*_{\mu}-2}u, \quad u>0 \; \text{in}\;
\Om, \quad u = 0 \; \mbox{on}\; \partial\Om,
\quad
\end{equation*}
where $\la >0,\; 0 < q < 1, $ $ 0<\mu<n$ and  $2^*_\mu=\frac{2n-\mu}{n-2}$.  Problems of the type $(P_\la)$  are inspired by the Hardy-Littlewood-Sobolev inequality:
\begin{equation}\label{new1}
\left(\int_{\mb R^n} \int_{\mb R^n}\frac{|u(x)|^{2^*_{\mu}}|u(y)|^{2^*_{\mu}}}{|x-y|^{\mu}}\mathrm{d}x\mathrm{d}y\right)^{\frac{1}{2^*_{\mu}}} \leq C ^{\frac{1}{2^*_{\mu}}}|u|^2_{2^*}, \; \text{for all} \; u\in D^{1,2} (\mathbb{R}^n).\end{equation}
where $C=C(n,\mu)$ is a positive constant and $2^{*}=\frac{2n}{n-2}.$ Recently, researchers are paying lot of attention to Choquard type equations and as a result, good amount of work has been done in this topic.  Existence of solutions for the equation of the type
\[ -\De u + w(x)u = (I_\alpha * |u|^p)|u|^{p-2}u \; \text{in} \; \mb R^n,\]
where $w(x)$ is an appropriate function, $I_\alpha$ is Reisz potential and $p>1$ is chosen appropriately,
have been studied in \cite{clsa, ghsc, lu, mosc, wang, ztxs}.  In \cite{lb}, Lieb  proved the existence and uniqueness, up to translations, of the ground state for the problem
\[ -\De u+ u = (|x|^{\mu} * F(u))f(u) \; \text{in}\; \mb R^n, \]
where  $f(t)$ is critical growth nonlinearity such that $|tf(t)| \leq C||t|^2 + |t|^{\frac{2n-\mu}{n-2s}}|$ for $t \in \mb R$, $\mu >0$,
some constant $C>0$ and $F(t) = \int_0^{z}f(z) \mathrm{d}z$. In \cite{ myang3,myang, myang1},  Gao and Yang  showed existence and multiplicity results for Brezis-Nirenberg type  problem of the nonlinear  Choquard equation
\begin{equation*}
-\De u = \left( \int_{\Om}\frac{|u|^{2^*_{\mu}}}{|x-y|^{\mu}}\mathrm{d}y \right)|u|^{2^*_{\mu}-2}u + \la g(u) \; \text{in}\;
\Om, \quad u = 0 \; \mbox{on}\; \partial\Om,
\end{equation*}
where $\Om$ is smooth bounded domain in $\mb R^n$, $n>2$, $\la>0$,
$0<\mu<n$ and $g(u)$ is a nonlinearity with certain assumptions. We
also cite \cite{AF,alves} and references therein for recent works on
Choquard equations.
 On the other hand, authors in \cite{hirano1} studied the existence of multiple solutions of the equation
\begin{equation}\label{hr}
-\De u = \la u^{-q}+ u^p,\; u>0 \; \text{in}\;
\Om, \quad u = 0 \; \mbox{on}\; \partial\Om,
\end{equation}
where $\Om$ is smooth bounded domain in $\mb R^n$, $n\geq 1$, $p>1$, $\la>0$  and $0<q<1$.  \noi The paper by Crandal, Rabinowitz and Tartar \cite{crt} is the starting point on semilinear problem with singular nonlinearity. A lot of work has been done related to existence and multiplicity results on singular nonlinearity, see \cite{haitao, hirano2, hirano1}. In \cite{haitao}, Haitao studied the equation \eqref{hr} for $n\geq 3$, $1<p\leq 2^{*}-1$ and showed the  existence  of two positive  solutions for maximal interval of the parameter $\la$ using monotone iterations and mountain pass lemma. Semilinear equations with singular nonlinearities has been also discussed in \cite{AJ, boc,guo,torres,ysy,yij}. Existence of multiple positive solutions for an elliptic equation with singular nonlinearity and sign changing weight functions has been shown in \cite{Hsulin}. There are many works on singular problem for equations involving $p$-Laplacian operator with critical  growth terms.  Among them we cite \cite{GR1,GR,GST,GST2,JPS,JS,  HFV, HKS} for readers and references therein. \\

\noi Observing these results, there arise a natural question that if problem $(P_\la)$ has multiple positive solutions, since $(P_\la)$ contains both singular as well as critical  nonlinearity in the sense of Hardy-Littlewood-Sobolev inequality \eqref{new1}. In this paper, we study the multiplicity results with convex-concave type critical growth and singular nonlinearity.
Here, we follow the approach as in the work of Hirano, Saccon and Shioji \cite{hirano1} and we would like to remark that the results proved here are new. The main difficulty in treating $(P_\la)$ is the presence of singular nonlinearity along with critical exponent in the sense of Hardy-Littlewood-Sobolev inequality which is nonlocal in nature. The energy functional no longer remains differentiable due to presence of singular nonlinearity, so usual minimax theorems are not applicable. Also the critical exponent term being nonlocal adds on the difficulty to study the Palais-Smale level around a nontrivial critical point. We obtain our results by studying the existence of minimizers that arise out of structure of Nehari manifold. The existence and multiplicity of solutions by  the method of Nehari manifold and fibering maps has been investigated in \cite{brzh,brwu, wu}.
For more details related to Nehari manifold, we refer to \cite{ brown, Chwu, wu2} and references therein. Moreover, under suitable assumptions, we obtained some regularity results. These results are obtained overcoming the standard bootstrap arguments.\\

\noi The paper is organized as follows: In section 2, we present some preliminaries required to prove our results and state the main results of our work. In section 3, we study the corresponding Nehari manifold using the fibering maps and properties of minimizers. Section 4 and 5 are devoted to show the existence of minimizers and solutions. Last but not least, in section 6, we show the regularity results for the solutions.

\section{Preliminaries and Main Results} In this section, we recall some preliminary results that are required in the later sections and also give the statement of our main results.
We denote $|\cdot|_p$ as the standard $L^p(\Om)$ norm, $1\leq p\leq \infty$ and $\|\cdot\|$ for $H^1_0(\Om)$ norm.
Now for each $\alpha \geq 0$, we set
\begin{equation*}\label{eq00}
 C_{\alpha} = \sup \left \{ \int_{\Om} |u|^\alpha \mathrm{d}x : \|u\| = 1\right \}.
 \end{equation*}
Then $C_0  = |\Om| $ = Lebesgue measure of $\Om$ and
$\int_{\Om} |u|^\alpha \mathrm{d}x \leq C_\alpha \|u\|^{\alpha}$, for all $ u \in H^1_0(\Om)$.
The key point to apply variational approach for the problem $(P_\la)$ is the following well-known Hardy-Littlewood-\textcolor{red}{Sobolev} inequality \cite{Lieb}.
 \begin{Proposition}\label{Hardy-littlewood}
 Let $t,r>1$ and $0<\mu<n $ with $1/t+\mu/n+1/r=2$, $f \in L^t(\mathbb R^n)$ and $h \in L^r(\mathbb R^n)$. There exists a sharp constant $C(t,n,\mu,r)$, independent of $f,h$ such that
 \begin{equation*}\label{har-lit}
\textcolor{red}{ \left|\int_{\mb R^n}\int_{\mb R^n} \frac{f(x)h(y)}{|x-y|^{\mu}}\mathrm{d}x\mathrm{d}y \right|}\leq C(t,n,\mu,r)|f|_t|h|_r.
 \end{equation*}
 \end{Proposition}
In general, let $f = h= |u|^q$ then
\[ \int_{\mb R^n}\int_{\mb R^n} \frac{|u(x)|^q|u(y)|^q}{|x-y|^{\mu}}\mathrm{d}x\mathrm{d}y\]
 is well defined if $|u|^q \in L^t(\mb R^n)$ for some $t>1$ satisfying
 $\ds \frac{2}{t}+ \frac{\mu}{n}=2.$
 Thus, for $u \in H^1(\mb R^n)$, by Sobolev Embedding theorems, we must have
 \[\frac{2n-\mu}{n}\leq q \leq \frac{2n-\mu}{n-2}.\]
 We say $(2n-\mu)/n$ is the lower critical exponent and $2^*_\mu = (2n-\mu)/(n-2)$ is the upper critical exponent in the sense of Hardy-Littlewood-Sobolev inequality. From this inequality, for each $u \in D^{1,2}(\mb R^n)$ we have
 \[\left(\int_{\mb R^n} \int_{\mb R^n}\frac{|u(x)|^{2^*_{\mu}}|u(y)|^{2^*_{\mu}}}{|x-y|^{\mu}}\mathrm{d}x\mathrm{d}y\right)^{\frac{1}{2^*_{\mu}}} \leq C(n,\mu)^{\frac{1}{2^*_{\mu}}}|u|^2_{2^*},\]
 where $C(n,\mu)$ is a suitable constant defined in Proposition \ref{Hardy-littlewood} and $2^* = 2n/(n-2)$. We define
 \begin{equation*}
 S_{H,L} := \inf\limits_{D^{1,2}(\mb R^n)\textcolor{red}{\setminus}\{0\}} \frac{\int_{\mb R^n}|\nabla u|^2 ~\mathrm{d}x}{\left(\int_{\mb R^n} \int_{\mb R^n}\frac{|u(x)|^{2^*_{\mu}}|u(y)|^{2^*_{\mu}}}{|x-y|^{\mu}}\mathrm{d}x\mathrm{d}y\right)^{\frac{1}{2^*_{\mu}}}}
 \end{equation*}
as the best constant which is achieved if and only if $u$ is of the form
\[C\left( \frac{t}{t^2+|x-x_0|^2}\right)^{\frac{n-2}{2}}, \; \; x \in \mb R^n,\]
for some $x_0 \in \mb R^n$, $C>0$ and $t>0$ (refer Lemma $1.2$ of \cite{myang}). Also, it satisfies
\begin{equation*}\label{critical-cho}
-\De u = \left( \int_{\mb R^n}\frac{|u|^{2^*_{\mu}}}{|x-y|^{\mu}}\mathrm{d}y \right)|u|^{2^*_{\mu}-2}u \; \; \text{in}\; \mb R^n
\end{equation*}
and it is well-known that this characterization of $u$ provides the minimizer for $S$, where
\[ S = \inf_{u \in H^{1}_0(\mb R^n)\setminus \{0\}} \frac{\int_{\mb R^n} |\nabla u|^2~\mathrm{d}x}{\left(\int_{\mb R^n}|u|^{2^*}\right)^{2/2^*}} .\]
We remark that $S_{H,L}$ does not depend on the domain $\Om$, see (\cite{myang}, Lemma $1.3$). Moreover, from Lemma $1.2$ of \cite{myang} we have
\begin{equation}\label{relation}
S_{H,L} = \frac{S}{C(n,\mu)^{\frac{1}{2^*_{\mu}}}}.
\end{equation}
\noi Consider the family of functions $\{U_{\epsilon}\}$ defined as
\[ U_{\epsilon}(x) = (n(n-2))^{\frac{n-2}{4}} \left( \frac{\epsilon}{\epsilon^2+ |x|^2}\right)^{\frac{n-2}{2}} \; \text{for}\; x \in \mb R^n \; \text{and} \; \epsilon >0.\]
Then for each $\e >0$, $U_\e$ satisfies
\begin{equation}\label{aubtal}
\int_{\mb R^n}|\nabla U_\e|^2 \mathrm{d}x= S^{n/2}.
\end{equation}
We recall the following results  from \cite{hirano1}.
Let $\Om^{\prime}$ be an open subset of $\Om$. Let $u,v$ be distributions on $\Om$, then we write
\[-\De u \leq v \; \text{in}\; \Om^{\prime}\]
if the inequality holds in the sense of distributions. If $u \in H^1_{\text{loc}}(\Om)$ and $u \in L^1_{\text{loc}}(\Om)$, it means that $- \int_{\Om} \nabla u \nabla \psi ~\mathrm{d}x \leq \int_{\Om} v\psi ~\mathrm{d}x $ for all $\psi \in C_c^{\infty}(\Om)$ with $\psi \geq 0$ and $\text{supp }\psi \subset \Om^{\prime}$. For regularity result, we require the following lemma (for proof, refer Theorem 8.15, \cite{GT}).
\begin{Theorem}\label{bddbel}
Let $\Om$ be a bounded domain in $\mb R^n$ with $n\geq 2$. Let $u \in H^1_0(\Om)$ and let $g \in L^{\alpha/2}(\Om)$ with $\alpha > n$ satisfying $- \De u \leq g$. Then $u$ is essentially bounded from above.
\end{Theorem}
Also, we define $\delta : \Om \rightarrow [0,\infty)$ by $\delta(x)=\inf\{|x-y|: y \in \partial \Om\}$, for each $x \in \Om$. For each $a>0$, we set $\Om_a = \{x\in \Om: \delta(x)<a\}$.
We fix $\Om$ to be a bounded domain in $\mb R^n$ with $n>2$, $0<q<1$ and $\la >0$, for rest of this paper.
\begin{Definition}
We say $u$ is a positive weak solution of $(P_\la)$
if $u>0$ in $\Om$, $u \in H^1_0(\Om)$ and
$$ \int_{\Om} (\nabla u \nabla \psi -\la u^{-q}\psi)~\mathrm{d}x  - \int_{\Om}\int_{\Om}\frac{|u(x)|^{2^*_{\mu}}|u(y)|^{2^*_{\mu}-2}u(y)\textcolor{red}{\psi}(y)}{|x-y|^{\mu}}~\mathrm{d}x\mathrm{d}y = 0 \;\; \text{for all} \; \psi \in C^{\infty}_c(\Om).$$
\end{Definition}
We define the functional $I_{\la} : H^1_{0}(\Om) \rightarrow (-\infty, \infty]$ by
\[ I_{\la}(u) = \frac12 \int_{\Om}|\nabla u|^2~ \mathrm{d}x- \frac{\la}{1-q} \int_\Om |u|^{1-q} \mathrm{d}x - \frac{1}{22^*_{\mu}}\int_{\Om}\int_{\Om}\frac{|u(x)|^{2^*_{\mu}}|u(y)|^{2^*_{\mu}}}{|x-y|^{\mu}}~\mathrm{d}x\mathrm{d}y, \]
for $u \in H^1_{0}(\Om)$. For each $0 < q <1$, we set
$\ds H_+ = \{ u \in H^1_0(\Om) : u \geq 0\}$ and
$$H_{+,q} = \{ u \in H_+ : u \not\equiv 0, |u|^{1-q} \in L^1(\Om)\} = H_+ \setminus \{0\} .$$
We recall the following Lemma $A.1$ of \cite{hirano1}.
\begin{Lemma}\label{lem2.1}
For each $w \in H_{+}$, there exists a sequence $\{w_k\}$ in $H^1_{0}(\Om)$ such that, $w_k \rightarrow w$ strongly in $H^1_0(\Om)$, where $0 \leq w_1 \leq w_2 \leq \ldots$ and $w_k$ has compact support in $\Om$, for each k .
\end{Lemma}
\noi For each $u\in H_{+,q}$ we define the fiber map $\phi_u:\mb R^+ \rightarrow \mb R$ by $\phi_u(t)=I_\la(tu)$. Then we prove the following:
\begin{Theorem} \label{mainthrm1}
Assume $0<q < 1$ and let $\Lambda$ be a constant defined by
\begin{equation*}
\begin{split}
\Lambda = & \sup \left\{\la >0:\text{ for each}  \; u\in H_{+,q}\backslash\{0\}, ~\phi_u(t)~  \text{has two critical points in}  ~(0, \infty)\right.\\
 &\left.\text{and}\; \sup\left\{ \int_{\Om}|\nabla u|^2~\mathrm{d}x\; : \; u \in H_{+,q}, \phi^{\prime}_u(1)=0,\;\phi^{\prime\prime}_u(1)>0 \right\} \leq (2^*_\mu S_{H,L}^{2^*_\mu})^{\frac{1}{2^*_\mu-1}} \right\}.
\end{split}
\end{equation*}
Then $\La >0$.
\end{Theorem}
Using the variational methods on the Nehari manifold, we will prove the following multiplicity result.
\begin{Theorem}\label{mainthrm2}
For all $\la \in (0, \Lambda)$, $(P_\la)$ has two positive weak solutions $u_\la$ and $v_\la$ in $C^\infty(\Om)\cap L^{\infty}(\Om)$.
\end{Theorem}
We also have that if $u$ is a positive weak solution of $(P_\la)$, then $u$ is a classical solution in the sense that $u \in C^\infty(\Om) \cap C(\bar \Om)$.
\begin{Theorem}\label{mainthrm3}
Let $u$ be a positive weak solution of $(P_\la)$, then there exist $K,\;L>0$ such that $L\delta \leq u \leq K\delta$ in $\Om$.
\end{Theorem}
\begin{Remark}
\textcolor{red}{\textbf{Doubly nonlocal problems: }We remark here that the doubly nonlocal problems of the the type
\begin{equation*}
(Q_\la)\quad \quad (-\De)^s u = \la u^{-q} + \left( \int_{\Om}\frac{|u|^{2^*_{\mu,s}}}{|x-y|^{\mu}}\mathrm{d}y \right)u^{2^*_{\mu,s}-1}, \quad u>0 \; \text{in}\;
\Om,\quad u = 0 \; \mbox{in}\; \mb R^n \setminus \Om,
\end{equation*}
where $0<s<1$ , $2^*_{\mu,s}= (2n-\mu)/(n-2s)$ and
$$ (-\De)^s u(x) = -\mathrm{P.V.}\int_{\mb R^n} \frac{u(x)-u(y)}{\vert x-y\vert^{n+2s}}\,\mathrm{d}y$$
({up to a normalizing constant}), $\mathrm{P.V.}$ denotes the Cauchy principal value, can be dealt in a similar manner to obtain the existence and multiplicity results. We refer \cite{TS2,TS1,TS3,TS4} in this context, where fractional laplacian with singular or choquard type nonlinearity has been studied. Also, we refer \cite{GMS,MR,repo1,repo2,ZZX} to readers and references therein. The underlying function space for $(Q_\la)$ is
$$X= \left\{u|\;u:\mb R^n \ra\mb R \;\text{is measurable},\;
u|_{\Om} \in L^2(\Om)\;
 \text{and}\;  \frac{(u(x)- u(y))}{ |x-y|^{\frac{n}{2}+s}}\in
L^2(Q)\right\},$$
\noi where $Q=\mb R^{2n}\setminus(\mc C\Om\times \mc C\Om)$ and
 $\mc C\Om := \mb R^n\setminus\Om$ endowed with the norm
 \[\|u\|_X = \|u\|_{L^2(\Om)} + \left( \int_{Q}\frac{|u(x)-u(y)|^{2}}{|x-y|^{n+2s}}\,\mathrm{d}x\mathrm{d}y\right)^{\frac12}.\]
 Then, $ X_0 = \{u\in X : u = 0 \;\text{a.e. in}\; \mb R^n\setminus \Om\}$ forms a Hilbert space with the norm $\|u\|=\left( \int_{Q}\frac{|u(x)-u(y)|^{2}}{|x-y|^{n+2s}}\,\mathrm{d}x\mathrm{d}y\right)^{\frac12}$. We say $u \in X_0$ is a weak solution of $(Q_\la)$ if
\begin{equation*}
\begin{split}
&\int_{Q}\frac{(u(x)-u(y))(\varphi(x)-\varphi(y))}{|x-y|^{n+2s}}~\mathrm{d}x\mathrm{d}y\\
&\quad = \la \int_{\Om}u^{-q}\varphi ~dx + \int_{\Om}\int_{\Om}\frac{|u(y)|^{2^*_{\mu,s}}|u(x)|^{2^*_{\mu,s}-2}u(x)\varphi(x)}{|x-y|^{\mu}}~\mathrm{d}y\mathrm{d}x,
\end{split}
\end{equation*}
for every $\varphi \in C_c^\infty(\Om)$. The energy functional $J_{\la}: X_0 \rightarrow \mb R^n$ is
\[J_{\la}(u):= \frac{\|u\|^2}{2} -\frac{\la}{1-q}\int_{\Om}|u|^{1-q}\mathrm{d}x- \frac{1}{22^*_{\mu,s}}\int_{\Om}\int_{\Om}\frac{|u(x)|^{2^*_{\mu,s}}|u(y)|^{2^*_{\mu,s}}}{|x-y|^{\mu}}~\mathrm{d}x\mathrm{d}y
.\]
For each $u\in X_{+,q}$, we define the fiber map $\phi_u:\mb R^+ \rightarrow \mb R$ by $\phi_u(t)=J_\la(tu)$. Now using the similar ideas as in section 3 one can define the Nehari manifold  $\mathcal{N}_\la$ and its decompositions $\mathcal{N}^{+}_{\la}$ and $\mathcal{N}^{-}_{\la}$. By minimising $J_\la$ over $\mathcal{N}^{+}_{\la}$, we can show the existence of first solution exactly as in  Proposition \ref{minattain1} and Proposition \ref{prp4.2}. To obtain the second solution, we can use the minimizers of
 \begin{equation*}
 S^H_s := \inf\left\{ \int_{\mb R^{2n}}\frac{|u(x)-u(y)|^2}{|x-y|^{n+2s}}{\,\mathrm{d}x\mathrm{d}y}\;; \; u \in {H^s(\mb R^n)},\;\int_{\mb R^{2n}}\frac{|u(x)|^{2^*_{\mu,s}}|u(y)|^{2^*_{\mu,s}}}{|x-y|^{\mu}}\mathrm{d}x\mathrm{d}y=1\right\}
 \end{equation*}
which are of the form $C\left( \frac{t}{t^2+|x|^2}\right)^{\frac{n-2s}{2}}, \; \; \text{for all}\; x \in \mb R^n$, where $C>0$ is constant and $t>0$. The asymptotic estimates involving these minimizers for Palais Smale analysis can be found in \cite{TS4}.
Following the   Proposition \ref{minattain2} and Proposition \ref{prp5.4} one can show the existence of minimizer of $J_\la$ over  $\mathcal{N}^{-}_{\la}$. The regularity of weak solutions of $(Q_\la)$ as obtained in Theorem \ref{mainthrm3} is not clear in this case. We leave this as an open question.}
\end{Remark}

\section{Nehari Manifold and Fibering Map analysis}
We denote $I_{\la} = I$ for simplicity. In this section, we describe the structure of Nehari Manifold associated to the functional $I$.  One can easily verify that the energy functional $I$ is not bounded below on the space $H^1_0(\Om)$. But we will show that $I$ is bounded below on this Nehari manifold and we will extract solutions by minimizing the functional on suitable subsets. The Nehari manifold is defined as
\[ \mc N_{\la} = \{ u \in H_{+,q} | \left\langle I^{\prime}(u),u\right\rangle = 0 \}. \]
\begin{Theorem}
$I$ is coercive and bounded below on $\mc N_{\la}$.
\end{Theorem}
\begin{proof} Since $u\in \mc N_\la$, using the embedding of $H^1_0(\Om)$ in $L^{1-q}(\Om)$, we obtain
\begin{equation*}
\begin{split}
I(u) & = \left(\frac{1}{2}-\frac{1}{22^{*}_{s}}\right)\|u\|^2- \la \left(\frac{1}{1-q}-\frac{1}{22^*_s}\right)\int_{\Om}|u|^{1-q}\mathrm{d}x\\
& \geq c_1 \|u\|^2 - c_2 \|u\|^{1-q}
\end{split}
\end{equation*}
for some positive constants $c_1$ and $c_2$. Thus, $I$ is coercive and bounded below on $\mc N_{\la}$.
\QED
\end{proof}

\noi From the definition of fiber map $\phi_u$, we have
$$\phi_u(t)=
        \frac{t^2}{2} \|u\|^2 - \frac{t^{1-q}}{1-q} \int_{\Om} |u|^{1-q} \mathrm{d}x - \frac{t^{22^*_{\mu}}}{22^*_{\mu}}\int_{\Om}\int_{\Om}\frac{|u(x)|^{2^*_{\mu}}|u(y)|^{2^*_{\mu}}}{|x-y|^{\mu}}~\mathrm{d}x\mathrm{d}y$$
for $u \in H^1_0(\Om)\textcolor{red}{\setminus \{0\}}$ and $t>0$, which gives
\[ \phi^{\prime}_u(t) = t \|u\|^2 -\la t^{-q} \int_{\Om}|u|^{1-q} \mathrm{d}x - t^{22^*_{\mu}-1}\int_{\Om}\int_{\Om}\frac{|u(x)|^{2^*_{\mu}}|u(y)|^{2^*_{\mu}}}{|x-y|^{\mu}}~\mathrm{d}x\mathrm{d}y \]
\[\text{and }\; \phi^{\prime \prime}_u(t) = \|u\|^2 + q \la t^{-q-1} \int_{\Om}|u|^{1-q} \mathrm{d}x - (22^*_\mu-1) t^{22^*_{\mu}-2}\int_{\Om}\int_{\Om}\frac{|u(x)|^{2^*_{\mu}}|u(y)|^{2^*_{\mu}}}{|x-y|^{\mu}}~\mathrm{d}x\mathrm{d}y .\]
It is easy to see that the points in $\mc N_{\la}$ are corresponding to critical points of $\phi_{u}$ at $t=1$. So, it is natural to divide $\mc N_{\la}$ in three sets corresponding to local minima, local maxima and points of inflexion. Therefore, we define
\begin{align*}
\mc N_{\la}^{+} = & \{ u \in \mc N_{\la}|~ \phi^{\prime}_u (1) = 0,~ \phi^{\prime \prime}_u(1) > 0\}=  \{ t_0u \in \mc N_{\la} |\; t_0 > 0,~ \phi^{\prime}_u (t_0) = 0,~ \phi^{\prime \prime}_u(t_0) > 0\},\\
\mc N_{\la}^{-} = & \{ u \in \mc N_{\la} |~ \phi^{\prime}_u (1) = 0,~ \phi^{\prime \prime}_u(1) <0\}=  \{ t_0u \in \mc N_{\la} |\; t_0 > 0, ~ \phi^{\prime}_u (t_0) = 0, ~\phi^{\prime \prime}_u(t_0) < 0\}
\end{align*}
and $ \mc N_{\lambda}^{0}= \{ u \in \mc N_{\la} | \phi^{\prime}_{u}(1)=0, \phi^{\prime \prime}_{u}(1)=0 \}. $
From Lemma 2.3 of \cite{myang}, we know that
\[\|\cdot\|_0 = \left( \int_\Om \int_\Om \frac{|u(x)|^{2^*_\mu}|u(y)|^{2^*_\mu}}{|x-y|^\mu}~\mathrm{d}x\mathrm{d}y\right)^{\frac{1}{22^*_\mu}}\]
defines a norm on $L^{2^*}(\Om)$. But in the next lemma, we will show that $\|\cdot\|$ defines an equivalent norm on $L^{2^*}(\Om)$.
\begin{Lemma}\label{singeqnr}
For $n>2$ and $0<\mu <n$,
\[\|\cdot\|_0 = \left( \int_\Om \int_\Om \frac{|u(x)|^{2^*_\mu}|u(y)|^{2^*_\mu}}{|x-y|^\mu}~\mathrm{d}x\mathrm{d}y\right)^{\frac{1}{22^*_\mu}}\]
defines an equivalent norm on $L^{2^*}(\Om)$.
\end{Lemma}
\begin{proof}
We can easily show that $L^{2^*}(\Om)$ is a Banach space under the norm $\|\cdot\|_0$(proof can be sketched using the techniques to prove $L^p(\Om)$ is a Banach space with the usual $L^p$-norm). By Hardy-Littlewood-Sobolev inequality, we have
\[ \|u\|_0 \leq (C(n,\mu))^{\frac{1}{22^*_{\mu}}}|u|_{2^*}. \]
So, the identity map from $(L^{2^*}(\Om), \|\cdot\|_0)$ to $(L^{2^*}(\Om), |\cdot|_{2^*})$ is linear and bounded. Thus, by open mapping theorem, we obtain $\|\cdot\|_0$ is an equivalent norm with respect to the standard norm $|\cdot|_{2^*}$ on $L^{2^*}(\Om)$.\hfill{\QED}
\end{proof}

\begin{Lemma}\label{lem3.2}
There exist $\la_*>0$ such that for each $u\in H_{+,q}\backslash\{0\}$, there is unique $t_1$ and $t_2$ with the property that $t_1<t_2$, $t_1 u\in \mc N_{\la}^{+}$ and $ t_2 u\in \mc N_{\la}^{-}$, for all $\la \in (0,\la_*)$.
\end{Lemma}
\begin{proof}
Define $A(u)= \int_{\Om}|u|^{1-q}~dx$ and $B(u)= \int_{\Om}\int_{\Om}\frac{|u(x)|^{2^*_{\mu}}|u(y)|^{2^*_{\mu}}}{|x-y|^{\mu}}~\mathrm{d}x\mathrm{d}y$. Let $u \in H_{+,q}\setminus \{0\}$ then we have
\begin{align*}
\frac{d}{dt}I(tu)
=& t\|u\|^2 -\la t^{-q} A(u) - t^{22_{\mu}^{*}-1} B(u)\\
=&t^{-q} \left (m_u(t) - \la A(u) \right ),
\end{align*}
where we define $m_u(t) := t^{1+q} \|u\|^2 - t^{22_{\mu}^{*}-1+q} B(u)$. Suppose $tu \in \mc N_\la$, then $\phi_{tu}^{\prime \prime}(1) = t^{2-q}m^{\prime}_u(t)$ and so $tu \in \mc N^+_\la$ (or $\mc N^-_\la$) provided $m^{\prime}_u(t)>0$ (or $<0$). Since $\ds \lim_{t \rightarrow \infty} m_u(t) = - \infty$,
we can easily see that $m_u(t)$ attains its maximum at $t_{max} = \left [ \frac{(1+q)\|u\|^2}{(22_{\mu}^{*}-1+q) B(u)} \right]^{\frac{1}{22_{\mu}^{*}-2}} $ and
\[ m_u(t_{max}) = \left( \frac{22_{\mu}^{*}-2}{22_{\mu}^{*}-1+q} \right) \left( \frac{1+q}{22_{\mu}^{*}-1+q} \right)^{\frac{1+q}{22_{\mu}^{*}-2}} \frac{(\|u\|^2)^\frac{22_{\mu}^{*}-1+q}{22_{\mu}^{*}-2}}{(B(u))^\frac{1+q}{22_{\mu}^{*}-2}}. \]
Clearly for $t> 0$, $tu \in \mc N_{\la}$  if and only if  $t$ is a solution of
\begin{equation}\label{tsol}
m_u(t) =  \la A(u).
\end{equation}
 So if $\la >0$ is sufficiently large, \eqref{tsol} has no solution and thus $\phi_u$ has no critical points. Hence, no multiple of $u$ lies in $\mc N_\la$. We already have $\phi_u(t) \rightarrow -\infty$ as $t \rightarrow +\infty$. On the other hand, if $\la>0$ sufficiently small, say $\la < \la_*$, then there exist exactly two solutions $t_1< t_{\max} < t_2$ of \eqref{tsol} with $m^{\prime}_u(t_1)>0$ and $m^{\prime}_u(t_2)<0$. Thus, there are exactly two multiples of $u \in \mc N_\la$ namely $t_1u \in \mc N^+_\la$ and $t_2u \in \mc N^-_\la$. It follows that $\phi_u$ has exactly two critical points- a local minimum at $t=t_1$ and a local maximum at $t=t_2$. Moreover, $\phi_u$ is decreasing in $(0,t_1)$, increasing in $(t_1,t_2)$ and decreasing in $(t_2,\infty)$. It remains to find the threshold $\la_*$ and for this, using Lemma \ref{singeqnr} we see that
\begin{equation*}
\begin{split}
&m_u(t_{max}) - \la A(u)\\
&\geq  \left( \frac{22_{\mu}^{*}-2}{22_{\mu}^{*}-1+q} \right) \left( \frac{1+q}{22_{\mu}^{*}-1+q} \right)^{\frac{1+q}{22_{\mu}^{*}-2}} \frac{(\|u\|^2)^\frac{22_{\mu}^{*}-1+q}{22_{\mu}^{*}-2}}{(B(u))^\frac{1+q}{22_{\mu}^{*}-2}} - \la C_{1-q} \|u\|^{1-q}\\
&\geq  \left( \frac{22_{\mu}^{*}-2}{22_{\mu}^{*}-1+q} \right) \left( \frac{1+q}{22_{\mu}^{*}-1+q} \right)^{\frac{1+q}{22_{\mu}^{*}-2}} \frac{K(\|u\|^2)^{\frac{1-q}{2}}}{\left(C(n,\mu)C_{2^*}^{\frac{22^*_\mu}{2^*}}\right)^{\frac{1+q}{22^*_\mu-2}}}- \la C_{1-q} \|u\|^{1-q}~ >0
\end{split}
\end{equation*}
if and only if $\la <  \left( \frac{22_{\mu}^{*}-2}{22_{\mu}^{*}-1+q} \right) \left( \frac{1+q}{22_{\mu}^{*}-1+q} \right)^{\frac{1+q}{22_{\mu}^{*}-2}} \left(C(n,\mu)C_{2^*}^{\frac{22^*_\mu}{2^*}}\right)^\frac{-1-q}{22_{\mu}^{*}-2} K^{-1}({C_{1-q}})^{-1} = \la_* $(say), where $C_{1-q}$ and $C(n,\mu)$ are defined in section 2 and $K$ is an appropriate positive constant.\QED
\end{proof}

\noi{\bf Proof of Theorem \ref{mainthrm1}}:
From \textcolor{red}{Lemma} \ref{lem3.2}, we see that $\Lambda \geq \la_* > 0$. Therefore, $\Lambda$ is positive.
\QED
\begin{Corollary}
$\mc N_{\la}^{0} = \{0\}$ for all $ \la \in (0, \La)$.
\end{Corollary}
\begin{proof}
Let $u \not\equiv 0 \in \mc N_{\la}^{0}$. Then $u \in \mc N_{\la}^{0}$ implies $u \in \mc N_{\la}$ that is, $1$ is a critical point of $\phi_{u}$. Using previous result, we say that $\phi_{u}$ has critical points corresponding to local minima or local maxima. So, $1$ is the critical point corresponding to local minima or local maxima of $\phi_{u}$. Thus, either $u \in \mc N_{\la}^{+}$ or $u \in \mc N_{\la}^{-}$ which is a contradiction.\QED
\end{proof}

\noi We can show that $\mc N_{\la}^{+}$ and $\mc N_{\la}^{-}$ are bounded in the following way:

\begin{Lemma}\label{le01}
The following holds:
\begin{enumerate}
\item[$(i)$] $\sup \{ \|u\|: u \in \mc N_{\la}^{+}\} < \infty $
\item[$(ii)$] $\inf \{ \|v\|: v \in \mc N_{\la}^{-} \} >0$  and $ \sup \{ \|v\| : v \in \mc N_{\la}^{-} , I(v) \leq M\} < \infty$ for each $M > 0$.
\end{enumerate}
Moreover, $\inf I(\mc N_{\la}^{+}) > - \infty$ and $\inf I(\mc N_{\la}^{-}) > - \infty$.
\end{Lemma}
\begin{proof}
\begin{enumerate}
\item[$(i)$] Let $u \in \mc N_{\la}^{+}$. Then we have
\begin{equation*}
\begin{split}
0 & < \phi^{\prime \prime}_u(1) =
(2-22_{\mu}^{*}) \|u\|^2 + \la (22_{\mu}^{*}-1+q) \int_{\Om}|u|^{1-q} \mathrm{d}x   \\
 & \leq (2-22_{\mu}^{*}) \|u\|^2 + \la (22_{\mu}^{*}-1+q) C_{1-q}\|u\|^{1-q} .\\
\end{split}
\end{equation*}
 Thus we obtain
\[ \|u\| \leq \left ( \frac{\la (22_{\mu}^{*}-1+q) C_{1-q}}{22_{\mu}^{*}-2} \right )^{\frac{1}{1+q}}\]
which implies that $\sup \{ \|u\| : u \in \mc N_{\la}^{+}\} < \infty $.
\item[$(ii)$] Let $v \in \mc N_{\la}^{-}$. We have
\begin{equation*}
\begin{split}
0 & > \phi^{\prime \prime}_v(1) =
(2-22_{\mu}^{*}) \|v\|^2 + \la (22_{\mu}^{*}-1+q) \int_{\Om}|v|^{1-q} \mathrm{d}x   \\
 & \geq (2-22_{\mu}^{*}) \|v\|^2 + \la (22_{\mu}^{*}-1+q) C_{1-q}\|v\|^{1-q} .\\
\end{split}
\end{equation*}
 Thus we obtain
\[ \|v\| \geq \left ( \frac{\la (22_{\mu}^{*}-1+q) C_{1-q}}{22_{\mu}^{*}-2} \right )^{\frac{1}{1+q}}\]
which implies that $\inf \{ \|v\| : v \in \mc N_{\la}^{-} \} >0$. If $v \in \mc N_\la^-$ and $I(v) \leq M$, we get
\[  \frac{(22_{\mu}^{*}-2)}{4 \times 2_{\mu}^{*}} \|v\|^2 - \la \left( \frac{22_{\mu}^{*}-1+q}{22_{\mu}^{*}(1-q)} \right)  C_{1-q}\|v\|^{1-q} \leq M \]
 which implies  $ \sup \{ \|v\| : v \in \mc N_{\la}^{-} , Iv \leq M\} < \infty$, for each $M > 0$. Also if $u \in \mc N_{\la}^{+}$, using Lemma $2.3$ of \cite{myang} we have
\[ I(u) \geq -\frac{(1+q)}{2(1-q)}  \|u\|^2 - \frac{(22_{\mu}^{*}-1+q)}{22_{\mu}^{*}(1-q)} K C_{2^{*}} \|u\|^{2^{*}}\]
for some positive constant $K$. So, using ($i$) we conclude that $\inf I(\mc N_{\la}^{+}) > - \infty$ and similarly, using ($ii$) we can show that $\inf I(\mc N_{\la}^{-}) > - \infty$.\QED
\end{enumerate}
\end{proof}

\begin{Lemma}\label{le03} Suppose $u$ and $v$ are minimizers of $I$ over $\mc N_{\la}^{+}$ and $\mc N_{\la}^{-}$ respectively. Then for each $ w \in H_{+}$,
\begin{enumerate}
\item  there exists $\epsilon_0 > 0$ such that $I(u +\epsilon w) \geq I(u)$ for each $ \epsilon \in [0, \epsilon_0]$, and
\item $t_{\epsilon} \rightarrow 1$  as $\epsilon \rightarrow 0^+$, where $t_{\epsilon}$ is the unique positive real number satisfying $t_{\epsilon} (v + \epsilon w) \in \mc N_{\la}^{-}.$
\end{enumerate}
\end{Lemma}
\begin{proof}\begin{enumerate}
\item{
Let $w \in H_{+}$ that is $w \in H^1_0(\Om)$ and $w \geq 0$. We set
\begin{equation*}
\begin{split}
\rho(\epsilon) = &\|u+\epsilon w\|^2 + \la q \int_{\Om}|u+\epsilon w|^{1-q}~\mathrm{d}x\\
 & \quad - (22^{*}_{\mu}-1) \int_{\Om}\int_{\Om}\frac{|(u+\epsilon w)(x)|^{2^*_{\mu}}|(u+\epsilon w)(y)|^{2^*_{\mu}}}{|x-y|^{\mu}}~\mathrm{d}x\mathrm{d}y
 \end{split}
 \end{equation*}
for each $\epsilon \geq 0$. Then using continuity of $\rho$ and the fact that $\rho(0) = \phi^{\prime \prime}_u(1) >0$ since $u \in \mc{N}^{+}_{\la}$, there exist $\epsilon_0>0 $ such that $\rho(\epsilon)> 0$ for $\epsilon \in [0, \epsilon_0]$. Since for each $\epsilon > 0$, there exists $t_{\epsilon}^{\prime}>0$ such that $t_{\epsilon}^{\prime}(u + \epsilon w) \in \mc{N}^{+}_{\la}$, so $t_{\epsilon}^{\prime} \rightarrow 1$ as $\epsilon \rightarrow 0$ and for each $\epsilon \in [0, \epsilon_0]$, we have
\[ I(u + \epsilon w) \geq I(t_{\epsilon}^{\prime}(u + \epsilon w))\geq \inf I(\mc{N}^{+}_{\la})= I(u).\]
}
\item{
We define $h :(0, \infty)\times \mb R^3 \rightarrow \mb R  $ by
\[ h(t,l_1,l_2,l_3) = l_1t - \la t^{-q}l_2 - t^{22^*_\mu-1}l_3 \]
for $(t,l_1,l_2,l_3)\in (0, \infty)\times \mb R^3.$ Then, $h$ is a $C^{\infty}$ function. Also, we have
\begin{equation*}
\frac{dh}{dt}\left(1, \|v\|^2,\int_{\Om}|v|^{1-q}~\mathrm{d}x , \int_{\Om}\int_{\Om}\frac{|v(x)|^{2^*_{\mu}}|v(y)|^{2^*_{\mu}}}{|x-y|^{\mu}}~\mathrm{d}x\mathrm{d}y\right) = \phi^{\prime \prime}_v(1)<0  \text{ and }\;
\end{equation*}
\begin{equation*}
 h\left(t_{\epsilon}, \|v+\epsilon w\|^2, \int_{\Om}|v + \epsilon w|^{1-q}~\mathrm{d}x , \int_{\Om}\int_{\Om}\frac{|(v + \epsilon w)(x)|^{2^*_{\mu}}|(v + \epsilon w)(y)|^{2^*_{\mu}}}{|x-y|^{\mu}}~\mathrm{d}x\mathrm{d}y\right) =0,
 \end{equation*}
 for each $\epsilon > 0$. Moreover,
\[ h\left(1, \|v\|^2, \int_{\Om}|v|^{1-q}~\mathrm{d}x, \int_{\Om}\int_{\Om}\frac{|v(x)|^{2^*_{\mu}}|v(y)|^{2^*_{\mu}}}{|x-y|^{\mu}}~\mathrm{d}x\mathrm{d}y\right) = \phi^{\prime}_v(1) = 0.\]
Therefore, by implicit function theorem, there exists an open neighborhood $ A \subset (0, \infty)$ and $B \subset \mb R^3$ containing $1$ and $\left(\|v\|^2,\int_{\Om}|v|^{1-q}~\mathrm{d}x, \int_{\Om}\int_{\Om}\frac{|v(x)|^{2^*_{\mu}}|v(y)|^{2^*_{\mu}}}{|x-y|^{\mu}}~\mathrm{d}x\mathrm{d}y \right)$ respectively such that  for all $y \in B$, $h(t,y) = 0 $ has a unique solution $ t = g(y)\in A $, where $g : B \rightarrow A$ is a continuous function. So, $$\left(\|v+\epsilon w\|^2,\; \int_{\Om}|v+ \epsilon w|^{1-q}~\mathrm{d}x,\int_{\Om}\int_{\Om}\frac{|(v + \epsilon w)(x)|^{2^*_{\mu}}|(v + \epsilon w)(y)|^{2^*_{\mu}}}{|x-y|^{\mu}}~\mathrm{d}x\mathrm{d}y\right) \in B$$
\[ \text{and, }g \left(\|v+\epsilon w\|^2,\; \int_{\Om}|v+ \epsilon w|^{1-q}~\mathrm{d}x,\int_{\Om}\int_{\Om}\frac{|(v + \epsilon w)(x)|^{2^*_{\mu}}|(v + \epsilon w)(y)|^{2^*_{\mu}}}{|x-y|^{\mu}}~\mathrm{d}x\mathrm{d}y\right) = t_{\epsilon},  \]
\noi since $h\left(t_{\epsilon},\|v+\epsilon w\|^2,\; \int_{\Om}|v+ \epsilon w|^{1-q}~\mathrm{d}x,\int_{\Om}\int_{\Om}\frac{|(v + \epsilon w)(x)|^{2^*_{\mu}}|(v + \epsilon w)(y)|^{2^*_{\mu}}}{|x-y|^{\mu}}~\mathrm{d}x\mathrm{d}y\right) = 0$. Thus, by continuity of $g$, we obtain $t_{\epsilon} \rightarrow 1$ as $\epsilon \rightarrow 0^+$.}\QED
\end{enumerate}
\end{proof}
\begin{Lemma} Suppose $u$ and $v$ are minimizers of $I$ on $\mc N_{\la}^{+}$ and $\mc N_{\la}^{-}$ respectively. Then
for each $w \in H_{+}$, we have $u^{-q}w, v^{-q} w \in L^{1}(\Om)$ and
\begin{align}
&\int_{\Om} (\nabla u \nabla w-\la u^{-q}w)~\mathrm{d}x  - \int_{\Om}\int_{\Om}\frac{|u(y)|^{2^*_{\mu}}|u(x)|^{2^*_{\mu}-2}u(x)w(x)}{|x-y|^{\mu}}~\mathrm{d}y\mathrm{d}x \geq 0 , \label{upos}\\
&\int_{\Om} (\nabla v \nabla w -\la v^{-q}w)~\mathrm{d}x  - \int_{\Om}\int_{\Om}\frac{|v(y)|^{2^*_{\mu}}|u(x)|^{2^*_{\mu}-2}v(x)w(x)}{|x-y|^{\mu}}~\mathrm{d}y\mathrm{d}x \geq 0.\label{vpos}\end{align}
Particularly, $u,\; v >0$ almost everywhere in $\Om$.
\end{Lemma}
\begin{proof}
Let $w \in H_{+}$. For sufficiently small $\epsilon > 0$, by Lemma \ref{le03} we get
\begin{equation}\label{solpos}
\begin{split}
0  \leq \frac{I(u+\epsilon w) - I(u)}{\epsilon}
  & =\frac{1}{2\e} (\|u+\epsilon w\|^2- \|u\|^2) - \frac{\la}{\e(1-q)} \int_{\Om}( |u + \epsilon w|^{1-q} - |u|^{1-q})~\mathrm{d}x \\
& - \frac{1}{ 22^*_\mu\e}\int_{\Om}\int_{\Om}\frac{|(u+\e w)(x)|^{2^*_{\mu}}|(u+\e w)(y)|^{2^*_{\mu}}- |u(x)|^{2^*_{\mu}}|u(y)|^{2^*_{\mu}} }{|x-y|^{\mu}}\mathrm{d}y\mathrm{d}x .\\
\end{split}
\end{equation}
We can easily verify that
\begin{enumerate}
\item[($i$)] $ \frac{1}{2\e} (\|u+\epsilon w\|^2- \|u\|^2) \rightarrow \int_{\Om} \nabla u \nabla w ~\mathrm{d}x \;\; \text{as}\; \epsilon \rightarrow 0^+$,
\item[($ii$)]  As $\epsilon \rightarrow 0^+$,
 \begin{align*} \frac{1}{ 22^*_\mu\e}&\int_{\Om}\int_{\Om}\frac{|(u+\e w)(x)|^{2^*_{\mu}}|(u+\e w)(y)|^{2^*_{\mu}}- |u(x)|^{2^*_{\mu}}|u(y)|^{2^*_{\mu}} }{|x-y|^{\mu}}\mathrm{d}y\mathrm{d}x \rightarrow \\
    & \int_\Om \int_\Om \frac{|u(y)|^{2^*_\mu}|u(x)|^{2^*_\mu-1}w(x)}{|x-y|^\mu}\mathrm{d}y\mathrm{d}x.
    \end{align*}
\end{enumerate}
Also we can see that for each $x \in \Om$, $\frac{|(u+\e w)(x)|^{1-q}- |u(x)|^{1-q}}{\e(1-q)}$ increases monotonically as $\e \rightarrow 0^+$ and
$$\lim\limits_{\epsilon \downarrow 0} \frac{|(u+\e w)(x)|^{1-q}- |u(x)|^{1-q}}{\e(1-q)} =
\left\{
    \begin{array}{ll}
    0 & \mbox{if} \; w(x)=0 \\
    (u(x))^{-q} w(x) & \mbox{if} \; w(x)>0 , u(x) > 0\\
    \infty & \mbox{if}\; w(x) > 0 , u(x) =0.
    \end{array}
\right.$$
So using monotone convergence theorem, we obtain $u^{-q}w \in L^1(\Om)$. Letting $\epsilon \downarrow 0$ in both sides of \eqref{solpos}, we obtain \eqref{upos}.
Next, we will show these properties for $v$. For each $\epsilon >0 $, there exists $t_{\epsilon}>0$ such that $t_{\epsilon}(v+\epsilon w) \in \mc N^-_\la$. By \textcolor{red}{Lemma} \ref{le03}(2), for sufficiently small $\epsilon > 0$, there holds
\[ I(t_{\epsilon}(v+\epsilon w)) \geq I(v) \geq I(t_{\epsilon}v)\]
which implies $I(t_{\epsilon}(v+\epsilon w)) - I(v) \geq 0$ and thus, we have
\begin{align*}
\frac{\la}{(1-q)\e} & \int_{\Om} |v+\epsilon w|^{1-q} - |v|^{1-q}\mathrm{d}x \leq \frac{t_{\epsilon}^{1+q}}{2\e} (\|v+\epsilon w\|^2 - \|v\|^2)\\
 &- \frac{t^{2^*_\mu}_\epsilon-1+q}{ 22^*_\mu \e}\int_{\Om}\int_{\Om}\frac{|(v+\e w)(x)|^{2^*_{\mu}}|(v+\e w)(y)|^{2^*_{\mu}}- |v(x)|^{2^*_{\mu}}|v(y)|^{2^*_{\mu}} }{|x-y|^{\mu}}\mathrm{d}y\mathrm{d}x.
 \end{align*}
As $\epsilon \downarrow 0$, $t_\epsilon \rightarrow 1$. Thus, using similar arguments as above, we obtain $v^{-q}w \in L^1(\Om)$ and \eqref{vpos} follows. \QED
\end{proof}

\section{Existence of minimizer on $\mc N_{\la}^{+}$ }
In this section, we will show that the minimum of $I$ is achieved in $\mc N_{\la}^{+}.$ Moreover, we show that this minimizer is also the first solution of $(P_\la).$
\begin{Proposition}\label{minattain1}
For all $\la\in (0,\La)$, there exist $u_\la \in \mc N_{\la}^{+}$ satisfying $I(u_\la) = \inf I(\mc N_{\la}^{+})$.
\end{Proposition}
\begin{proof}
Assume $\la \in (0, \Lambda)$. Let $\{u_{k}\} \subset \mc N_{\la}^{+}$ be a sequence such that $I(u_{k}) \rightarrow \inf I(\mc N_{\la}^{+})$ as $k \rightarrow \infty$. Using \textcolor{red}{Lemma} \ref{le01}, we can assume that there exist $u_\la$ such that $u_{k} \rightharpoonup u_\la$ weakly as $k \rightarrow \infty$ in $H^1_0(\Om)$. First we will show that $\inf I(\mc N_{\la}^{+}) < 0$. Let $u_0 \in \mc N_{\la}^{+}$, then we have $\phi^{\prime \prime}_{u_0}(1) >0$ which gives
\[  (1+q)\|u_0\|^2 > (22^{*}_{\mu}-1+q) \int_{\Om}\int_{\Om}\frac{|u(x)|^{2^*_{\mu}}|u(y)|^{2^*_{\mu}}}{|x-y|^{\mu}}~\mathrm{d}y\mathrm{d}x.\]
Therefore, using $2^{*}_{\mu}-1>0$ we obtain
\begin{equation*}
\begin{split}
I(u_0) & = \left( \frac12 - \frac{1}{1-q} \right) \|u_0\|^2 + \left( \frac{1}{1-q} - \frac{1}{22^{*}_{\mu}} \right)\int_{\Om} |u_0|^{2^{*}_{s}}dx\\
& \leq -  \frac{(1+q)}{2(1-q)} \|u_0\|^2 + \frac{(1+q)}{22^{*}_{\mu}(1-q)} \|u_0\|^2= \left( \frac{1}{22^*_\mu}-\frac{1}{2}\right)\left(\frac{1+q}{1-q}\right) \|u_0\|^2<0.\\
\end{split}
\end{equation*}\\
This shows that $\inf I(\mc N_{\la}^{+}) < 0$.  We set $w_k := (u_k - u_\la)$ and claim that $u_k \rightarrow u_\la$ strongly as $k \rightarrow \infty$ in $H^1_0(\Om)$. Suppose $\|w_k\|^2 \rightarrow c^2 \neq 0$ and $\int_{\Om}\int_{\Om}\frac{|w_k(y)|^{2^*_{\mu}}|w_k(x)|^{2^*_{\mu}}}{|x-y|^{\mu}}~\mathrm{d}y\mathrm{d}x \rightarrow d^{22^*_\mu}$ as $k \rightarrow \infty$. Using Brezis-Lieb lemma and Lemma \textcolor{red}{2.2} of \cite{myang}, we have
\[\|u_k\|^2 = \|w_k\|^2 + \|u_\la\|^2 + o_k(1), \; \text{and}\]
\begin{equation*}
\begin{split}
\int_{\Om} \int_{\Om} \frac{|u_k(x)|^{2^*_{\mu}}|u_k(y)|^{2^*_{\mu}}}{|x-y|^{\mu}}~\mathrm{d}x\mathrm{d}y
& = \int_{\Om} \int_{\Om} \frac{|w_k(x)|^{2^*_{\mu}}|w_k(y)|^{2^*_{\mu}}}{|x-y|^{\mu}}~\mathrm{d}x\mathrm{d}y\\
 &+\int_{\Om} \int_{\Om} \frac{|u_\la(x)|^{2^*_{\mu}}|u_\la(y)|^{2^*_{\mu}}}{|x-y|^{\mu}}~\mathrm{d}x\mathrm{d}y +o_k(1).
\end{split}
\end{equation*}
Since $u_k \in \mc N^+_{\la}$, we obtain
\begin{equation}\label{eq_8}
0  = \lim_{k \rightarrow \infty} \phi^{\prime}_{u_k}(1) = \phi^{\prime}_{u_\la}(1)+c^2 -d^{22^*_\mu}
\end{equation}
which implies
$$ \|u_\la\|^2+c^2 = \la \int_{\Om}|u_\la|^{1-q}\mathrm{d}x + \int_{\Om}\int_{\Om}\frac{|u_\la(x)|^{2^*_{\mu}}|u_\la(y)|^{2^*_{\mu}}}{|x-y|^{\mu}}~\mathrm{d}y\mathrm{d}x   +d^{22^*_\mu}.$$
We claim that $u_\la \in H_{+,q}$. Suppose $u_\la\equiv 0$. If $c=0$, then $0 > \inf I(\mc N^+_{\la}) = I(0)=0$, which is a contradiction and if $c\neq 0$, then
\begin{equation}\label{eq4.2new}
\inf I(\mc N_{\la}^{+})= I(0)+\frac{c^2}{2} - \frac{d^{22^*_\mu}}{2^*_\mu} = \frac{c^2}{2} - \frac{d^{22^*_\mu}}{2^*_\mu} .\end{equation}
But from the definition of $S_{H,L}$, we have $c^2 \geq S_{H,L}d^2 $. Also from \eqref{eq_8}, we have $c^2=d^{22^*_\mu}$. Then \eqref{eq4.2new} implies
\[ 0 > \inf I(\mc N_{\la}^{+}) = \left(\frac{1}{2}-\frac{1}{22^*_\mu}\right)c^2 \geq \left(\frac{n-\mu+2}{2(2n-\mu)}\right) S_{H,L}^{\frac{2n-\mu}{n-\mu+2}},\]
which is again a contradiction. Thus, $u_\la \in H_{+,q}$. So, there exists $0 < t_1 < t_2$ such that $\phi^{\prime}_{u_{\la}}(t_1)= \phi^{\prime}_{u_{\la}}(t_2) = 0$ and $t_{1}u_{\la} \in \mc N^{+}_{\la}$. Then, three cases arise: \\
(i) $t_2 < 1$,\\
(ii) $t_2 \geq 1$ and $\frac{c^2}{2}- \frac{d^{22^*_\mu}}{22^*_\mu} < 0$, and \\
(iii) $t_2 \geq 1$ and $\frac{c^2}{2}- \frac{d^{22^*_\mu}}{22^*_\mu} \geq 0$.\\
\noi Case (i) Let $h(t) = \phi_{u_{\la}}(t)+ \frac{c^2t^2}{2} - \frac{d^{22^*_\mu}t^{22^*_\mu}}{22^*_\mu}$, for $t >0$. By \eqref{eq_8}, we obtain $ h^{\prime}(1) = \phi^{\prime}_{u_{\la}}(1)+c^2-d^{22^*_\mu} = 0$ and
\begin{equation*}
 h^{\prime}(t_2) = \phi^{\prime}_{u_{\la}}(t_2)+t_2c^2-{t_2}^{22^*_\mu-1}d^{22^*_\mu} = {t_2}(c^2 - {t_2}^{22^*_\mu-2}d^{22^*_\mu}) \geq {t_2}(c^2 - d^{\textcolor{red}{22^*_\mu}})
 > 0
\end{equation*}
which implies that $h$ increases on $[t_2,1]$. Then we get
\begin{equation*}
\begin{split}
\inf I( \mc N_{\la}^{+}) &= \lim I(u_k) \geq \phi_{u_{\la}}(1) + \frac{c^2}{2}- \frac{d^{22^*_\mu}}{22^*_\mu}  = h(1) > h(t_2)\\
& =\phi_{u_{\la}}(t_2) + \frac{c^2t_{2}^{2}}{2}- \frac{d^{22^*_\mu}t_{2}^{22^*_\mu}}{\textcolor{red}{22^*_\mu}} \geq \phi_{u_{\la}}(t_2) + \frac{t_{2}^{2}}{2} (c^2 - d^{22^*_\mu})\\
& > \phi_{u_{\la}}(t_2) > \phi_{u_{\la}}(t_1) \geq \inf I( \mc N_{\la}^{+}),
\end{split}
\end{equation*}
which is a contradiction.\\
\noi Case (ii) In this case, since $\la \in (0, \Lambda)$, $(c^2/2 - d^{22^*_\mu}/{(22^*_\mu)}) < 0$ and $S_{H,L}d^2 \leq c^2$, we have
\[ \sup\{ \|u\|^2: u \in \mc N^+_{\la}\} \leq (2^*_\mu S_{H,L}^{2^*_\mu})^{\frac{1}{2^*_\mu-1}} < c^2 \leq \sup\{ \|u\|^2: u \in \mc N^+_{\la}\},\]
which gives a contradiction. Consequently, only case (iii) holds and  we have
$$ \inf I(\mc N_{\la}^{+}) = I(u_\la)+\frac{c^2}{2}-\frac{d^{22^*_\mu}}{22^*_\mu} \geq I(u_\la) = \phi_{u_\la}(1) \geq \phi_{u_\la}(t_1) \geq \inf I(\mc N^+_{\la}) .$$
Clearly, this holds only when $t_1 = 1$ and $(c^2/2 - d^{22^*_\mu}/{22^*_\mu}) = 0$ which yields $c=0$ and  $u_k \rightarrow u_\la$ strongly as $k \rightarrow \infty$ in $H^1_0(\Om)$. Thus, $u_\la \in \mc N^+_{\la}$ and $I(u_\la) = \inf I(\mc N^+_\la)$. \QED
\end{proof}
\begin{Proposition}\label{prp4.2}
$u_\la$  is a positive weak solution of $(P_\la)$.
\end{Proposition}
\begin{proof}
Let $\psi \in C^{\infty}_c(\Om)$. By Lemma \ref{L-infty} and Lemma \ref{solpos}, since $u_\la > 0$, we can find $ \alpha >0$ such that $u_\la \geq \alpha$ on support of $\psi$. Then $u+\e \psi\ge0$ for small $\e$.  With similar reasoning as  in the proof of Lemma \ref{le03}, we can show that $ I(u_\la+\epsilon \psi) \geq I(u_\la)$ for sufficiently small $\epsilon >0$. Then we have
\begin{equation*}
\begin{split}
0 & \leq \lim \limits_{\epsilon \rightarrow 0} \frac{I(u_\la+\epsilon\psi) - I(u_\la)}{\epsilon}\\
&= \int_\Om \nabla u \nabla \psi ~\mathrm{d}x - \la \int_\Om u_{\la}^{-q}\psi ~\mathrm{d}x - \int_{\Om}\int_{\Om}\frac{|u(y)|^{2^*_{\mu}}|u(x)|^{2^*_{\mu}-2}u(x)\varphi(x)}{|x-y|^{\mu}}~\mathrm{d}y\mathrm{d}x.
\end{split}
\end{equation*}
Since $\psi \in C^{\infty}_c(\Om)$ is arbitrary, we conclude that $u_\la$ is a positive weak solution of $(P_\la)$.\QED
\end{proof}
\section{Existence of minimizer on $\mc N^{-}_{\la}$}

\noi In this section, we will show the existence of second solution by proving the existence of minimizer of $I$ on $\mc N^-_{\la}$. We need some lemmas to prove this and for instance, we assume $0 \in \Om$ and $B_\delta \subset \Om \subset B_{2\delta}$. We recall the definition of $U_\e$ from section 2. Let $\eta \in C_c^\infty (\Om)$ such that for all $x \in \mb R^n$, $0 \leq \eta(x) \leq 1$ and
$$\eta(x)=
\left\{
    \begin{array}{ll}
    1 & \mbox{if} \; x \in B_\delta \\
    0 & \mbox{if} \; x \in \mb R^n \setminus \Om.
    \end{array}
\right.$$
We define, for $\e >0$
\[\Phi_\e(x) := \eta(x)U_\e(x).\]
Moreover, since $u_\la$ is positive and bounded (see Lemma \ref{L-infty}), we can find $m,M>0$ such that for each $x \in \Om$, $m \leq u_\la(x)\leq M$.
\begin{Lemma}
 For each sufficiently small $\epsilon >0$,
 $$\sup \{I(u_\la + t\Phi_\epsilon): t\geq 0\} < I(u_\la) + \frac{n-\mu+2}{2(2n-\mu)} S_{H,L}^{\frac{2n-\mu}{n+2-\mu}}$$
\end{Lemma}
\begin{proof}
We assume $\epsilon>0$ to be sufficiently small. Since $\eta \equiv 1$ near $x=0$, using \eqref{aubtal} and \eqref{relation} we can find $r_1>0$ such that
\[\int_\Om |\nabla \Phi_\e|^2 \mathrm{d}x \leq S^{n/2}+ r_1\e^{n-2}= C(n,\mu)^{\frac{n(n-2)}{2(2n-\mu)}}S_{H,L}^{n/2}+ r_1 \e^{n-2}.\]
Also using inequality $3.9$ of \cite{myang}, we can find $r_2>0$ such that
\[{\int_{\Om}\int_{\Om} \frac{|\Phi_\e(y)|^{2^*_{\mu}}|\Phi_\e(x)|^{2^*_{\mu}}}{|x-y|^{\mu}}~\mathrm{d}y\mathrm{d}x}\geq C(n,\mu)^{n/2}S_{H,L}^{(2n-\mu)/2}- r_2 \e^{(2n-\mu)/2}.\]
 We now fix $1 < \rho <{n}/{(n-2)}$ and set $\delta= n(n-2)$, $\gamma_\eta = \sup\{|x| : x \in \text{supp }\eta\}$,
 $$r_3= \delta^{(n-2)\rho/4} \int_{|x| \leq \gamma_\eta} |x|^{-(n-2)\rho}~\mathrm{d}x, \; \text{and}\; r_4 = (\delta/4)^{(n+2)/4}\int_{|x|\leq 1}~\mathrm{d}x.$$
 Then we have
 $$ \int_{\Om}|\Phi_\e|^\rho~ \mathrm{d}x \leq r_3 \e^{\frac{(n-2)\rho}{2}}.$$
 Next, we consider the integrals
 \begin{equation*}
 \int_{|x|\leq \e}\int_{|y|\leq \e} \frac{|\Phi_\e(y)|^{2^*_{\mu}}|\Phi_\e(x)|^{2^*_{\mu}-1}}{|x-y|^{\mu}}~\mathrm{d}y\mathrm{d}x \text{ and }  \int_{|x|\leq \e}\int_{|y|> \e} \frac{|\Phi_\e(y)|^{2^*_{\mu}}|\Phi_\e(x)|^{2^*_{\mu}-1}}{|x-y|^{\mu}}~\mathrm{d}y\mathrm{d}x
 \end{equation*}
 separately. Firstly, we see that
 \begin{equation*}
 \begin{split}
  \int_{|x|\leq \e}&\int_{|y|\leq \e}  \frac{|\Phi_\e(y)|^{2^*_{\mu}}|\Phi_\e(x)|^{2^*_{\mu}-1}}{|x-y|^{\mu}}~\mathrm{d}y\mathrm{d}x\\
  &=  \int_{|x|\leq \e}\int_{|y|\leq \e} \frac{(n(n-2))^{\frac{(n-2)(22^*_\mu-1)}{4}} \e^{\frac{(2-n)(22^*_\mu-1)}{2}}}{|x-y|^\mu (1+|\frac{x}{\e}|^2)^\frac{(n-2)(2^*_\mu-1)}{2}(1+|\frac{y}{\e}|^2)^\frac{(n-2)2^*_\mu}{2}}~\mathrm{d}y\mathrm{d}x\\
  & \geq \int_{|x|\leq \e}\int_{|y|\leq \e} \frac{(n(n-2))^{\frac{(n-2)(22^*_\mu-1)}{4}} \e^{\frac{(2-n)(22^*_\mu-1)}{2}-\mu}}{ (1+|\frac{x}{\e}|^2)^\frac{(n-2)(2^*_\mu-1)}{2}(1+|\frac{y}{\e}|^2)^\frac{(n-2)2^*_\mu}{2}}~\mathrm{d}y\mathrm{d}x\\
  & =  \int_{|x|\leq 1}\int_{|y|\leq 1} \frac{(n(n-2))^{\frac{(n-2)(22^*_\mu-1)}{4}} \e^{\frac{n-2}{2}}}{ (1+|x|^2)^\frac{(n-2)(2^*_\mu-1)}{2}(1+|y|^2)^\frac{(n-2)2^*_\mu}{2}}~\mathrm{d}y\mathrm{d}x = o\left(\e^{\frac{n-2}{2}}\right).
 \end{split}
 \end{equation*}
 Secondly, in a similar manner we get
  \begin{equation*}
 \begin{split}
  \int_{|x|\leq \e}\int_{|y|> \e} &\frac{|\Psi_\e(y)|^{2^*_{\mu}}|\Psi_\e(x)|^{2^*_{\mu}-1}}{|x-y|^{\mu}}~\mathrm{d}y\mathrm{d}x\\
  &=  \int_{|x|\leq \e}\int_{|y|>\e} \frac{(n(n-2))^{\frac{(n-2)(22^*_\mu-1)}{4}} \e^{\frac{(2-n)(22^*_\mu-1)}{2}}}{|x-y|^\mu (1+|\frac{x}{\e}|^2)^\frac{(n-2)(2^*_\mu-1)}{2}(1+|\frac{y}{\e}|^2)^\frac{(n-2)2^*_\mu}{2}}~\mathrm{d}y\mathrm{d}x\\
  & \geq \int_{|x|\leq \e}\int_{|y|> \e} \frac{(n(n-2))^{\frac{(n-2)(22^*_\mu-1)}{4}} \e^{\frac{(2-n)(22^*_\mu-1)}{2}}}{ (|y|+\e)^\mu (1+|\frac{x}{\e}|^2)^\frac{(n-2)(2^*_\mu-1)}{2}(1+|\frac{y}{\e}|^2)^\frac{(n-2)2^*_\mu}{2}}~\mathrm{d}y\mathrm{d}x\\
  & =  \int_{|x|\leq 1}\int_{|y|> 1} \frac{(n(n-2))^{\frac{(n-2)(22^*_\mu-1)}{4}} \e^{\frac{n-2}{2}}}{ (1+|x|^2)^\frac{(n-2)(2^*_\mu-1)}{2}(1+|y|^2)^\frac{(n-2)2^*_\mu}{2}(1+|y|)^\mu}~\mathrm{d}y\mathrm{d}x = o\left(\e^{\frac{n-2}{2}}\right).
 \end{split}
 \end{equation*}
 Therefore, we can easily find $r_4>0$ which is independent of $\e$ such that
 \[\int_{|x|\leq \e}\int_\Om \frac{|\Psi_\e(y)|^{2^*_{\mu}}|\Psi_\e(x)|^{2^*_{\mu}-1}}{|x-y|^{\mu}}~\mathrm{d}y\mathrm{d}x \geq r_4 \e^{\frac{n-2}{2}}.\]
 We can find appropriate constants $\rho_1, \rho_2 >0$ such that the following inequalities holds :
\begin{enumerate}
\item  $\displaystyle \la \left( \frac{(c+d)^{1-q}}{1-q} - \frac{c^{1-q}}{1-q} - \frac{d}{c^q} \right) \geq -\frac{\rho_1d^{\rho}}{r_3},$ for all $c \geq m, d\geq 0$,
\item For each $\e, \; t>0$,
\begin{equation*}
\begin{split} &\frac{1}{22^*_\mu}\int_\Om\int_\Om\left(\frac{|(u_\la+t\Psi_\e)(y)|^{2^*_\mu}|(u_\la+t\Psi_\e)(x)|^{2^*_\mu}}{|x-y|^\mu}- \frac{|u_\la(y)|^{2^*_\mu}|u_\la(x)|^{2^*_\mu}}{|x-y|^\mu}\right)~\mathrm{d}y \mathrm{d}x\\
&\quad \quad - \int_\Om\int_\Om \frac{|u_\la(y)|^{2^*_\mu} |u_\la(x)|^{2^*_\mu-2}u_\la(x)t \Psi_\e(x)}{|x-y|^\mu}~\mathrm{d}y\mathrm{d}x\\ &\geq \frac{t^{22^*_\mu}}{22^*_\mu} \int_\Om \int_\Om \frac{|\Psi_\e(x)|^{2^*_\mu}|\Psi_\e(y)|^{2^*_\mu}}{|x-y|^\mu}~\mathrm{d}y\mathrm{d}x,
\end{split}
\end{equation*}
\item
For each $\e>0$, $0\leq u_\la(x)\leq M$ and $t\Psi_\e(x)\geq 1$
\begin{equation*}
\begin{split}
& \frac{1}{22^*_\mu}\int_\Om\int_\Om\left(\frac{|(u_\la+t\Psi_\e)(y)|^{2^*_\mu}|(u_\la+t\Psi_\e)(x)|^{2^*_\mu}}{|x-y|^\mu}- \frac{|u_\la(y)|^{2^*_\mu}|u_\la(x)|^{2^*_\mu}}{|x-y|^\mu}\right)~\mathrm{d}y \mathrm{d}x\\
&\quad  \quad - \int_\Om\int_\Om \frac{|u_\la(y)|^{2^*_\mu} |u_\la(x)|^{2^*_\mu-2}u_\la(x)t \Psi_\e(x)}{|x-y|^\mu}~\mathrm{d}y\mathrm{d}x\\
& \geq \frac{t^{22^*_\mu}}{22^*_\mu} \int_\Om \int_\Om \frac{|\Psi_\e(x)|^{2^*_\mu}|\Psi_\e(y)|^{2^*_\mu}}{|x-y|^\mu}~\mathrm{d}y\mathrm{d}x + \frac{\rho_2 t^{22^*_\mu-1}}{ (22^*_\mu-1)} \int_\Om\int_\Om \frac{|\Psi_\e(y)|^{2^*_{\mu}}|\Psi_\e(x)|^{2^*_{\mu}-1}}{|x-y|^{\mu}}~\mathrm{d}y\mathrm{d}x
\end{split}
\end{equation*}
\end{enumerate}
Since $u_\la$ is a positive weak solution of $(P_\la)$, using above inequalities, we obtain
 {\small \begin{equation*}
 \begin{split}
 &I(u_\la +t\Phi_\epsilon)-I(u_\la)\\
 & = I(u_\la+t\Phi_\epsilon)-I(u_\la)\\
  & \quad\quad - t\left( \int_{\Om} (\nabla u_\la \nabla \Phi_{\textcolor{red}{\epsilon}} -\la u_\la^{-q}\Phi_{\textcolor{red}{\epsilon}})~\mathrm{d}x -
 \int_{\Om}\int_{\Om}\frac{|u_\la(y)|^{2^*_{\mu}}|u_\la(x)|^{2^*_{\mu}-2}u_\la(x)\Phi_{\textcolor{red}{\epsilon}}(x)}{|x-y|^{\mu}}~\mathrm{d}y\mathrm{d}x\right)\\
 &= \frac{t^2}{2} \int_\Om |\nabla \Phi_\e|^2~\mathrm{d}x- \frac{1}{22^*_\mu}\int_\Om\int_\Om\left(\frac{|(u_\la+t\Phi_\e)(y)|^{2^*_\mu}|(u_\la+t\Phi_\e)(x)|^{2^*_\mu}}{|x-y|^\mu}- \frac{|u_\la(y)|^{2^*_\mu}|u_\la(x)|^{2^*_\mu}}{|x-y|^\mu}\right)~\mathrm{d}y \mathrm{d}x\\
 &- \int_\Om\int_\Om \frac{|u_\la(y)|^{2^*_\mu} |u_\la(x)|^{2^*_\mu-2}u_\la(x)t \Phi_\e(x)}{|x-y|^\mu}~\mathrm{d}y\mathrm{d}x - \la \int_{\Om}\left( \frac{|u_\la+t\Phi_\epsilon|^{1-q}-|u_\la|^{1-q}}{1-q}~\mathrm{d}x - t\Phi_\epsilon u_{\la}^{-q} \right)\mathrm{d}x\\
 & \leq \frac{t^2}{2}\left(C(n,\mu)^{\frac{n(n-2)}{2(2n-\mu)}}S_{H,L}^{n/2}+r_1\epsilon^{n-2}\right)-\frac{t^{22^*_\mu}}{22^*_\mu}
 \int_\Om \int_\Om \frac{|\Phi_\e(x)|^{2^*_\mu}|\Phi_\e(y)|^{2^*_\mu}}{|x-y|^\mu}~\mathrm{d}y\mathrm{d}x + \frac{\rho_1t^{\rho}}{r_3}\int_{\Om}|\Phi_\epsilon|^{\rho}\mathrm{d}x\\
 & \leq \frac{t^2}{2}\left(C(n,\mu)^{\frac{n(n-2)}{2(2n-\mu)}}S_{H,L}^{n/2}+r_1\epsilon^{n-2}\right)
 -\frac{t^{22^*_\mu}}{22^*_\mu}\left(C(n,\mu)^{\frac{n}{2}} S_{H,L}^{(2n-\mu)/2}-r_2\epsilon^{(2n-\mu)/2}\right)
 +\rho_1t^{\rho}\epsilon^{(n-2)\rho/2}
 \end{split}
\end{equation*}}
 for $0\leq t\leq 1/2$. Since we can assume $t\Phi_\epsilon \geq 1$, for each $t \geq 1/2$ and $|x| \leq \e$, we have
 \begin{equation*}
 \begin{split}
 &I(u_\la+t\Phi_\epsilon)-I(u_\la)\\
 & \leq \frac{t^2}{2}\left(C(n,\mu)^{\frac{n(n-2)}{2(2n-\mu)}}S_{H,L}^{n/2}+r_1\epsilon^{n-2}\right)-\frac{t^{22^*_\mu}}{22^*_\mu}
 \int_\Om \int_\Om \frac{|\Phi_\e(x)|^{2^*_\mu}|\Phi_\e(y)|^{2^*_\mu}}{|x-y|^\mu}~\mathrm{d}y\mathrm{d}x\\
 &\quad\quad - \frac{\rho_2 t^{22^*_\mu-1}}{ (22^*_\mu-1)} \int_\Om\int_\Om \frac{|\Phi_\e(y)|^{2^*_{\mu}}|\Phi_\e(x)|^{2^*_{\mu}-1}}{|x-y|^{\mu}}~\mathrm{d}y\mathrm{d}x+ \frac{\rho_1t^{\rho}}{r_3}\int_{\Om}|\Phi_\epsilon|^{\rho}\mathrm{d}x\\
 \end{split}
 \end{equation*}
 \begin{equation*}
 \begin{split}
 & \leq \frac{t^2}{2}\left(C(n,\mu)^{\frac{n(n-2)}{2(2n-\mu)}}S_{H,L}^{n/2}+r_1\epsilon^{n-2}\right)
 -\frac{t^{22^*_\mu}}{22^*_\mu}\left(C(n,\mu)^{\frac{n}{2}} S_{H,L}^{(2n-\mu)/2}-r_2\epsilon^{(2n-\mu)/2}\right)\\
 &\quad \quad-\frac{\rho_2 t^{22^*_\mu-1}}{22^*_\mu-1}\e^{(n-2)/2}+\rho_1t^{\rho}\epsilon^{(n-2)\rho/2}.
\end{split}
\end{equation*}
Now, we define a function $h_\epsilon : [0, \infty) \rightarrow \mb R$ by
\begin{align*}
h_\e(t)=\frac{t^2}{2}\left(C(n,\mu)^{\frac{n(n-2)}{2(2n-\mu)}}S_{H,L}^{n/2}+r_1\epsilon^{n-2}\right)
 &-\frac{t^{22^*_\mu}}{22^*_\mu}\left(C(n,\mu)^{\frac{n}{2}} S_{H,L}^{(2n-\mu)/2}-r_2\epsilon^{(2n-\mu)/2}\right)\\
& \quad +\rho_1t^{\rho}\epsilon^{(n-2)\rho/2}
 \end{align*}
 on the interval $[0,1/2)$ and
 \begin{align*}
 h_\e(t)& =\frac{t^2}{2}\left(C(n,\mu)^{\frac{n(n-2)}{2(2n-\mu)}}S_{H,L}^{n/2}+r_1\epsilon^{n-2}\right)
 -\frac{t^{22^*_\mu}}{22^*_\mu}\left(C(n,\mu)^{\frac{n}{2}} S_{H,L}^{(2n-\mu)/2}-r_2\epsilon^{(n-\mu)/2}\right)\\
 & \quad-\frac{\rho_2 t^{22^*_\mu-1}}{22^*_\mu-1}\e^{(n-2)/2}+\rho_1t^{\rho}\epsilon^{(n-2)\rho/2}
 \end{align*}
 on the interval $[1/2,\infty)$.
With some computations, it can be checked that $h_\epsilon$ attains its maximum at
\[ t = \left(\frac{C(n,\mu)^{\frac{n(n-2)}{2(2n-\mu)}}S_{H,L}^{n/2}+r_1\epsilon^{n-2}}{C(n,\mu)^{\frac{n}{2}} S_{H,L}^{(2n-\mu)/2}-r_2\epsilon^{(2n-\mu)/2}}\right)^{\frac{n-2}{2(n-\mu+2)}}- \frac{\rho_2 \e^{(n-2)/2}}{(22^*_\mu-1)C(n,\mu)^{\frac{n}{2}} S_{H,L}^{(2n-\mu)/2}} + o(\e^{(n-2)/2}), \]
Therefore we get
\begin{equation*}
\begin{split}
&\sup\{I(u_\la+t\Phi_\epsilon)-I(u_\la): t\geq 0\}\\
&\leq\frac{n-\mu+2}{2(2n-\mu)}\left(\frac{C(n,\mu)^{\frac{n(n-2)}{2(2n-\mu)}}S_{H,L}^{n/2}+r_1\epsilon^{n-2}}{\left(C(n,\mu)^{\frac{n}{2}} S_{H,L}^{(2n-\mu)/2}-r_2\epsilon^{(2n-\mu)/2}\right)^{(n-2)/(2n-\mu)}}\right)^{\frac{2n-\mu}{(n-\mu+2)}}\\
&\quad- \frac{\rho_2 \e^{(n-2)/2}C(n,\mu)^{\frac{n(n-2)}{2(2n-\mu)}}S_{H,L}^{n/2}}{(22^*_\mu-1)C(n,\mu)^{\frac{n}{2}} S_{H,L}^{(2n-\mu)/2}}+o(\epsilon^{(n-2)/2})\\
 &<\frac{n-\mu+2}{2(2n-\mu)} S_{H,L}^{\frac{2n-\mu}{n-\mu+2}} .
 \end{split}
 \end{equation*}
This completes the proof.\QED
\end{proof}
\begin{Lemma}
There holds $\inf I(\mc N^-_{\la}) < I(u_\la)+\frac{n-\mu+2}{2(2n-\mu)} S_{H,L}^{\frac{2n-\mu}{n-\mu+2}}$.
\end{Lemma}
\begin{proof}
We start by fixing sufficiently small $\epsilon >0$ as in the previous lemma and define functions $\sigma_1, \sigma_2: [0,\infty) \rightarrow \mb R$ by
\begin{align*}
\sigma_1(t) = \int_\Om |\nabla (u_\la+t\Psi_\e)|^2\mathrm{d}x &- \la \int_{\Om}|u_\la+t\Phi_\epsilon|^{1-q}~\mathrm{d}x\\
 &\quad-\int_{\Om}\int_{\Om}\frac{|\textcolor{red}{(u_\la+t\Psi_\e)}(y)|^{2^*_{\mu}}|\textcolor{red}{(u_\la+t\Psi_\e)}(x)|^{2^*_{\mu}}}{|x-y|^{\mu}}~\mathrm{d}y\mathrm{d}x,
 \end{align*}
\begin{align*}
 \text{and }\sigma_2(t) = \int_\Om |\nabla (u_\la+t\Psi_\e)|^2\mathrm{d}x &+\la q \int_{\Om}|u_\la+t\Phi_\epsilon|^{1-q}~\mathrm{d}x\\
  & \quad-(22^*_\mu-1) \int_{\Om}\int_{\Om}\frac{|\textcolor{red}{(u_\la+t\Psi_\e)}(y)|^{2^*_{\mu}}|\textcolor{red}{(u_\la+t\Psi_\e)}(x)|^{2^*_{\mu}}}{|x-y|^{\mu}}~\mathrm{d}y\mathrm{d}x,
\end{align*}
Let $t_0 = \sup\{t \geq 0: \sigma(t)\geq 0 \}$, then $\sigma_2(0) = \phi^{\prime\prime}_{u_\la}(1) >0$ and $\sigma_2(t) \rightarrow -\infty$ as $t \rightarrow \infty$ which implies $0<t_0<\infty$. As $\la \in (0, \Lambda)$, we obtain $\sigma_1(t_0)>0$ and since $\sigma_1(t) \rightarrow -\infty$ as $t \rightarrow \infty$, there exists $t^{\prime} \in (t_0, \infty)$ such that $\sigma_1(t^{\prime})=0.$ This gives $\phi^{\prime \prime}_{u_\la+t^{\prime}\Phi_\epsilon}(1)<0$, because $t^{\prime}>t_0$ which implies $\sigma_2(t^{\prime})<0.$ Hence, $ (u_\la+t^{\prime}\Phi_\epsilon)\in \mc N^-_{\la}$ and using previous lemma, we get
\[\inf I({\mc N^-_\la}) \leq I(u_\la+t\Psi_\e) < I(u) + \frac{n-\mu+2}{2(2n-\mu)} S_{H,L}^{\frac{2n-\mu}{n-\mu+2}}.\hspace{3.5cm}\text{ \QED}\]
\end{proof}

\begin{Proposition}\label{minattain2}
There exists $v_\la \in \mc N_{\la}^{-}$ satisfying $I(v_\la) = \inf I(\mc N_{\la}^{-})$.
\end{Proposition}
\begin{proof}
Let $\{v_k\}$ be a sequence in $\mc N^-_{\la}$ such that $I(v_k) \rightarrow \inf I(\mc N^-_{\la})$ as $k \rightarrow \infty$. Using lemma \ref{le01}, we may assume that there exist $v_\la \in H^1_0(\Om)$ such that $v_k \rightharpoonup v_\la$ weakly as $k \rightarrow \infty$ in $H^1_0(\Om)$. We set $z_k := (v_k - v_\la)$ and claim that $v_k \rightarrow v_\la$ strongly as $k \rightarrow \infty$ in $H^1_0(\Om)$. Suppose $\|z_k\|^2 \rightarrow c^2$ and ${\int_{\Om}\int_{\Om} \frac{|z_k(y)|^{2^*_{\mu}}|z_k(x)|^{2^*_{\mu}}}{|x-y|^{\mu}}~\mathrm{d}y\mathrm{d}x} \rightarrow d^{22^*_\mu}$ as $k \rightarrow \infty$. Using Brezis-Lieb lemma and Lemma 2.2 of \cite{myang}, we have
\[\|v_k\|^2 =  \|z_k\|^2 +\|v_\la\|^2 + o_k(1), \; \text{and}\]
\begin{equation*}
\begin{split}
\int_{\Om} \int_{\Om} \frac{|v_k(x)|^{2^*_{\mu}}|v_k(y)|^{2^*_{\mu}}}{|x-y|^{\mu}}~\mathrm{d}x\mathrm{d}y
& = \int_{\Om} \int_{\Om} \frac{|z_k(x)|^{2^*_{\mu}}|z_k(y)|^{2^*_{\mu}}}{|x-y|^{\mu}}~\mathrm{d}x\mathrm{d}y\\
 &+\int_{\Om} \int_{\Om} \frac{|v_\la(x)|^{2^*_{\mu}}|v_\la(y)|^{2^*_{\mu}}}{|x-y|^{\mu}}~\mathrm{d}x\mathrm{d}y +o_k(1).
\end{split}
\end{equation*}
Since $v_k \in \mc N^-_{\la}$, we obtain
\begin{equation}\label{eq8}
0  = \lim_{k \rightarrow \infty} \phi^{\prime}_{v_k}(1) = \phi^{\prime}_{v_\la}(1)+c^2 -d^{22^*_\mu}
\end{equation}
which implies
$$ \|v_\la\|^2+c^2 = \la \int_{\Om}|v_\la|^{1-q}\mathrm{d}x + \int_{\Om}\int_{\Om}\frac{|v_\la(x)|^{2^*_{\mu}}|v_\la(y)|^{2^*_{\mu}}}{|x-y|^{\mu}}~\mathrm{d}y\mathrm{d}x   +d^{22^*_\mu}.$$
We claim that $v_\la \in H_{+,q}$. Suppose $v_\la =0$, this implies $c \neq 0$ (using lemma \ref{le01}(ii)) and thus
\[ \inf I(\mc N^-_{\la}) = \lim I(v_k) = I(0)+\frac{c^2}{2}-\frac{d^{22^*_\mu}}{22^*_\mu} \geq \frac{n-\mu+2}{2(2n-\mu)}S_{H,L}^\frac{2n-\mu}{n-\mu+2},\]
as done in Lemma \ref{minattain1}. But by previous lemma, $ \inf I(\mc N^-_\la) < I(u_\la)+ \frac{n-\mu+2}{2(2n-\mu)}S_{H,L}^\frac{2n-\mu}{n-\mu+2}$ implying $\inf I(\mc N^+_\la)=I(u_\la) >0$, which is a contradiction. So $v_\la \in H_{+,q}$ and thus, our assumption $\la \in (0, \Lambda)$ says that there exists $0<t_1<t_2$ such that $\phi^{\prime}_{v_\la}(t_1)= \phi^{\prime}_{v_\la}(t_2)=0$ and $t_1 v_\la \in \mc N^+_{\la}$, $t_2 v_\la \in \mc N^-_{\la}$. Let us define $f,g :(0,\infty) \rightarrow \mb R$ by
\begin{equation}\label{eq11}
 f(t) = \frac{c^2t^2}{2}-\frac{d^{22^*_\mu}t^{22^*_\mu}}{22^*_\mu} \; \text{ and } \; g(t) = \phi_{v_\la}(t)+ f(t).
\end{equation}
Then, following three cases arise :\\
\noi (i) $t_2 < 1,$\\
\noi(ii) $t_2 \geq 1$ and $d >0$, and \\
\noi(iii) $t_2 \geq 1$ and $d = 0$.

\noi Case (i) $t_2 <1$ implies $g^{\prime}(1)= \phi^{\prime}_{v_\la}(1)+ f^{\prime}(1) = 0$, using $\eqref{eq11}$ and $g^{\prime}(t_2)= \phi^{\prime}_{v_\la}(t_2)+ f^{\prime}(t_2) = t_2(c^2-d^{22^*_\mu}t_{2}^{22^*_\mu-2}) \geq t_2(c^2-d^{22^*_\mu})>0$. This implies that $g$ is increasing on $[t_2,1]$ and we have
\[ \inf I(\mc N^-_{\la}) = g(1)> g(t_2) \geq I(t_2 v_\la)+\frac{{t_2^2}}{2}(c^2-d^{22^*_\mu}) > I(t_2 v_\la) \geq \inf I(\mc N^-_{\la}) \]
which is a contradiction.\\
\noi Case (ii) Let $\underline{t}=(c^2/d^{22^*_\mu})^{\frac{1}{22^*_\mu-2}}$ and we can check that $f$ attains its maximum at $\underline{t}$ and $$f(\underline{t})= \frac{c^2{\underline{t}}^2}{2}-\frac{d^{22^*_\mu}{\underline{t}}^{22^*_\mu}}{22^*_\mu}=\frac{n-\mu+2}{2(2n-\mu)}\left( \frac{c}{d}\right)^{\frac{22^*_\mu}{2^*_\mu-1}} \geq \frac{n-\mu+2}{2(2n-\mu)} S_{H,L}^{\frac{2n-\mu}{n-\mu+2}}. $$
Also, $f^{\prime}({t})= (c^2-d^{22^*_\mu}t^{22^*_\mu-2})t >0 $ if $0<t<\underline{t}$ and $f^{\prime}(t) <0$ if $t > \underline{t}$. Moreover, we know $g(1) = \max \limits_{t>0} \{g(t)\} \geq g(\underline{t})$ using the assumption $\la \in (0, \Lambda)$. If $\underline{t}\leq 1$, then we have
\[ \inf I(\mc N^-_{\la}) = g(1)\geq g(\underline{t})=I(\underline{t}v_\la)+ f(\underline{t}) \geq I(t_1 v_\la)+\frac{n-\mu+2}{2(2n-\mu)} S_{H,L}^{\frac{2n-\mu}{n-\mu+2}} \]
which contradicts the previous lemma. Thus, we must have $\underline{t}>1$. Since $g^{\prime}(t) \leq 0$ for $t\geq 1$, there holds $\phi^{\prime\prime}_{v_\la}(t) \leq -f^{\prime}(t) \leq 0 $ for $1\leq t \leq \underline{t}$. Then we have $\underline{t}\leq t_1$ or $t_2=1$. If $\underline{t}\leq t_1$ then
\[ \inf I(\mc N^-_{\la})= g(1) \geq  g(\underline{t})=I(\underline{t}v_\la)+ f(\underline{t}) \geq I(t_1v_\la)+ \frac{n-\mu+2}{2(2n-\mu)} S_{H,L}^{\frac{2n-\mu}{n-\mu+2}}\]
which is a contradiction. If $t_2=1$ then using $c^2=d^{22^*_\mu}$ we get
\[ \inf I(\mc N^-_{\la})=g(1)= I(v_\la)+ \left(\frac{c^2}{2}-\frac{d^{22^*_\mu}}{22^*_\mu}\right)\geq I(v_\la) + \frac{n-\mu+2}{2(2n-\mu)} S_{H,L}^{\frac{2n-\mu}{n-\mu+2}} \]
which is a contradiction and thus, only case (iii) holds. If $c \neq 0$, then $\phi^{\prime}_{v_\la}(1) = -c^2 <0$ and $\phi^{\prime \prime}_{v_\la}(1) = -c^2<0$ which contradicts $t_2 \geq 1$. Thus, $c=0$ which implies $v_k \rightarrow v_\la$ strongly as $k \rightarrow \infty$ in $H^1_0(\Om)$. Consequently, $v_\la \in \mc N^-_{\la}$ and $\inf I(\mc N^-_{\la})= I(v_\la)$.\QED
\end{proof}

\begin{Proposition}\label{prp5.4}
 For $\la \in (0,\La),$ $v_\la$ is a positive weak solution of $(P_\la)$.
\end{Proposition}
\begin{proof}
Let $\psi \in C^{\infty}_c(\Om)$. By Lemma \ref{L-infty} and \ref{solpos}, since $v_\la > 0$, we can find $ \alpha >0$ such that $v_\la \geq \alpha$ on support of $\psi$. Also, $t_{\epsilon} \rightarrow 1$ as $\epsilon \rightarrow 0+$, where $t_{\epsilon}$ is the unique positive real number corresponding to $(v_\la+\epsilon \psi)$ such that $t_\epsilon (v_\la+\epsilon \psi) \in \mc N^{-}_{\la}$. Then, by lemma \ref{le03} we have
\begin{equation*}
\begin{split}
0 & \leq \lim\limits_{\epsilon \rightarrow 0}\frac{I(t_\e(v_\la+\epsilon\psi)) - I(v_\la)}{\epsilon} \leq \lim \limits_{\epsilon \rightarrow 0}\frac{I(t_{\epsilon}(v_\la+\epsilon\psi)) - I(t_{\epsilon} v_\la)}{\epsilon}\\
& = \int_\Om \nabla v_\la \nabla \psi~\mathrm{d}x - \la \int_\Om  v_{\la}^{-q}\psi~\mathrm{d}x - \int_{\Om}\int_{\Om}\frac{|v_\la(y)|^{2^*_{\mu}}|v_\la(x)|^{2^*_{\mu}-2}v_\la(x)\varphi(x)}{|x-y|^{\mu}}~\mathrm{d}y\mathrm{d}x.
\end{split}
\end{equation*}
Since $\psi \in C^{\infty}_c(\Om)$ is arbitrary, we conclude that $v_{\la}$ is positive weak solution of $(P_\la)$.\QED
\end{proof}

\section{Regularity of  weak solutions}

In this section, we shall prove some regularity properties of positive weak solutions of $(P_{\la})$. We begin with the following lemma.

\begin{Lemma}\label{reg}
Suppose $u$ is a  weak solution of $(P_{\la})$, then for each $w \in H_0^1(\Om)$, it satisfies
$ u^{-q}w \in L^{1}(\Om)$ and
 \[\int_\Om \nabla u\nabla w~\mathrm{d}x- \la \int_\Om u^{-q}w~\mathrm{d}x - \int_{\Om}\int_{\Om}\frac{|u(y)|^{2^*_{\mu}}|u(x)|^{2^*_{\mu}-2}u(x)w(x)}{|x-y|^{\mu}}~\mathrm{d}y\mathrm{d}x = 0. \]
\end{Lemma}
\begin{proof}
Let $u$ be a  weak solution of $(P_{\la})$ and $w \in H_{+}$. By \textcolor{red}{Lemma} $\ref{lem2.1}$, we obtain a sequence $\{w_k \} \in H^1_{0}$ such that $\{w_k\} \rightarrow w$ strongly as $k \rightarrow \infty$ in $H^1_0(\Om)$, each $w_k$ has compact support in $
\Om$ and $0 \leq w_1 \leq w_2 \leq \ldots$ Since each $w_k$ has compact support in $\Om$ and $u$ is a positive weak solution of $(P_\la)$, for each $k$ we obtain
\[  \la \int_\Om  u^{-q}w_k~\mathrm{d}x =  \int_\Om \nabla u\nabla w_k~\mathrm{d}x-\int_{\Om}\int_{\Om}\frac{|u(y)|^{2^*_{\mu}}|u(x)|^{2^*_{\mu}-2}u(x)w_k(x)}{|x-y|^{\mu}}~\mathrm{d}y\mathrm{d}x  .\]
Using monotone convergence theorem, we obtain $u^{-q}w \in L^{1}(\Om)$  and
\[\la \int_\Om u^{-q}w~\mathrm{d}x = \int_\Om \nabla u\nabla w~\mathrm{d}x-\int_{\Om}\int_{\Om}\frac{|u(y)|^{2^*_{\mu}}|u(x)|^{2^*_{\mu}-2}u(x)w(x)}{|x-y|^{\mu}}~\mathrm{d}y\mathrm{d}x. \]
If $w \in H^1_0(\Om)$, then $w = w^+ - w^-$ and $w^+, w^- \in H_{+}$. Since we proved the lemma for each $w \in H_+$, we obtain the conclusion.\QED
\end{proof}
\begin{Lemma}
Let $u$ be a positive weak solution of $(P_{\la})$. Then $u \in L^p(\Om)$, for each $p \in [1,\infty)$.
\end{Lemma}
\begin{proof}
From proof of Lemma $6.1$ of \cite{myang}, we have
\begin{equation}\label{chobdd}
\int_\Om \frac{|u(y)|^{2^*_\mu}}{|x-y|^\mu}~\mathrm{d}y \in L^\infty(\Om).
\end{equation}
We claim that $u \in L^{2\beta}(\Om)$ implies $u \in L^{2^*\beta}(\Om)$, for each $\beta \in [1,\infty)$. So, let us assume $u \in L^{2\beta}(\Om)$ with $
\beta \in [1,\infty)$. Let $K>0$ and set \textcolor{red}{$\psi = \min\{u^{\beta-1},K\}$}. Then, $u\psi,\; u\psi^2 \in H^1_0(\Om)$. Using Lemma \ref{reg}, \textcolor{red}{ there exist a constant $M>0$ such that for each $R>0$,} we get
\begin{equation*}
\begin{split}
\int_{\Om}|\nabla u\psi|^2~\mathrm{d}x  & \leq \beta \int_\Om \nabla u \nabla (u\psi^2)~ \mathrm{d}x\\
 &= \beta  \left(\la\int_\Om u^{-q}u\psi^2~\mathrm{d}x + \int_\Om\int_\Om \frac{|u(y)|^{2^*_\mu}|u(x)|^{2^*_\mu-2}u(x)(u\psi^2)(x)}{|x-y|^\mu}~\mathrm{d}y\mathrm{d}x \right)\\
 &\leq  \beta \left( \int_\Om (\la u^{-q}u\psi^2+ M u^{2^*_\mu-2}u^2\psi^2 )~\mathrm{d}x\right)\\
 & \leq \beta\la\int_\Om u^{2\beta-1+q}~\mathrm{d}x + \beta M\left(\int_{u\leq R} u^{2\beta-2+2^*_\mu}~\mathrm{d}x + \int_{u>R} u^{2^*_\mu-2}u^2\psi^2\right)\\
 & \leq k_1+\beta\textcolor{red}{M} |\Om|R^{2\beta-2+2^*_\mu}+ k_2 \left(\int_{u>R} u^{\frac{n(2^*_\mu-2)}{2}}~\mathrm{d}x \right)^{\frac{2}{n}} \left( \int_{u>K}(u^2\psi^2)^{\frac{n}{n-2}}~\mathrm{d}x \right)^{\textcolor{red}{\frac{n-2}{2}}}\\
& \leq k_1+\beta\textcolor{red}{M} |\Om|R^{2\beta-2+2^*_\mu}+ k_3 \left(\int_{u>R} u^{\frac{n(2^*_\mu-2)}{2}}~\mathrm{d}x \right)^{\frac{2}{n}} \left( \int_\Om |\nabla(u\psi)|^2 ~\mathrm{d}x \right),
\end{split}
\end{equation*}
where $k_1,\;k_2$ and $k_3$ are positive constants independent of both $K$ and $R$. We can appropriately chose $R$ such that
\[\textcolor{red}{k_3\left(\int_{u>R} u^{\frac{n(2^*_\mu-2)}{2}}~\mathrm{d}x \right)^{\frac{2}{n}} \leq \frac12}.\]
Then we get
\[\int_{u^{\beta-1}\leq K}|\nabla u^\beta|^2~\mathrm{d}x \leq \int_\Om |\nabla(u\psi)|^2~\mathrm{d}x \leq 2(k_1+\beta\textcolor{red}{M} |\Om|R^{2\beta-2+2^*_\mu} ).\]
Now, letting $K \rightarrow \infty$, we get
\[\int_\Om|\nabla u^\beta|^2~\mathrm{d}x \leq 2(k_1+\beta |\Om|R^{2\beta-2+2^*_\mu} ).\]
This implies $u^\beta \in H^1_0(\Om)$ and therefore by Sobolev embedding theorem, we get $  u \in L^{2^*\beta}(\Om)$. Finally, using an inductive argument, we can say $u \in L^p(\Om)$, for each $1\leq p< \infty$.\QED
\end{proof}
\begin{Lemma}\label{L-infty}
Each positive weak solution of $(P_\la)$ belongs to $L^\infty(\Om)$.
\end{Lemma}
\begin{proof}
Let $u$ be a positive weak solution of $(P_\la)$. Then, for each $0 \leq \psi \in C^\infty_c(\Om)$ we have
\begin{equation*}
\begin{split}
\int_\Om \nabla(u-1)^+\nabla \psi~\mathrm{d}x &\leq \int_\Om \left( \la  + \left( \int_\Om \frac{|u(y)|^{2^*_\mu}}{|x-y|^\mu}~\mathrm{d}y\right)|u(x)|^{2^*_\mu-2}u(x)\right)\psi(x)~\mathrm{d}x\\
& \leq \int_\Om \left( \la  + M|u(x)|^{2^*_\mu-2}u(x)\right)\psi(x)~\mathrm{d}x,
\end{split}
\end{equation*}
where $M>0 $ is positive constant. Since $2^*_\mu -1 > n/2$, we use Theorem \ref{bddbel} to conclude that $(u-1)$ is bounded from above. Therefore, $u \in L^\infty(\Om)$. \QED
\end{proof}
\noi\textcolor{red}{{\bf Proof of Theorem \ref{mainthrm2}:} Now the proof of Theorem \ref{mainthrm2} follows from Proposition \ref{prp4.2}, Proposition \ref{prp5.4} and Lemma \ref{L-infty}.\QED}

\noi Before proving our next result, let us recall Proposition $3$ and Lemma A.2 of \cite{hirano1}. They show that the sufficient condition for the assumptions in Theorem \ref{distrel} are satisfied. We denote $B(x,r)$ the open ball of radius $r>0$ centered at $x \in \mb R^n$.
\begin{Proposition}
Assume that there exists $R>0$ such that for each $x \in \partial \Om$, there is $y \in \mb R^n$ with $|x-y|=R$ and $B(y,R)\cap \Om = \emptyset$. Then $\De \delta \leq (n-1)/R$. In particular, if $\Om$ is convex, then $\De \delta \leq 0$.
\end{Proposition}
\begin{Lemma}\label{dist1}
The function $\delta$ is Fr$\acute{e}$chet differentiable almost everywhere in $\Om$ and $|\nabla \delta|=1$ at which $\delta$ is Fr$\acute{e}$chet differentiable. Moreover, the first order derivatives of $\delta$ in the sense of distributions and those in classical sense coincide.
\end{Lemma}
Next we need  the following to show the regularity upto the boundary. We follow \cite{hirano1}.
\begin{Theorem}\label{distrel}
Let us assume that there exist $a\textcolor{red}{>} 0$, $R \geq 0$ and $q \leq s <1$ such that
\[\De \delta \leq R\delta^{-s} \; \text{in} \; \Om_a,\]
where $\delta$ is defined in section 2.
Then there exist $K>0$ such that $u \leq K\delta$ in $\Om$, where $u$ is a positive weak solution of $(P_\la)$.
\end{Theorem}
\begin{proof}
Let $u$ be a positive weak solution of $(P_\la)$. Using $u>0$ and Lemma \ref{L-infty}, for each $x \in \Om$, we get
\begin{equation}\label{distrel1}
\begin{split}
\la u^{-q}(x)+ \int_\Om \frac{|u(y)|^{2^*_\mu}}{|x-y|^\mu}~\mathrm{d}y |u(x)|^{2^*_\mu-2}u(x) \leq \la u^{-q}(x)+k_1 \leq k_2  u^{-q}(x),
\end{split}
\end{equation}
for some positive constant $k_1$ and $k_2$. We choose $\bar h \in \mb R^n$ such that $\bar h > \|u\|_\infty$. Let us define
$$\varrho_0(t) =
\left\{
    \begin{array}{ll}
    \bar h(2t-t^{2-s}) & \mbox{if} \; \textcolor{red}{t}\in[0,1] \\
    \bar h & \mbox{if} \; \textcolor{red}{t} \geq 1.
    \end{array}
\right.$$
Then for each $0<t<1$, we get
\[ \varrho_0^{\prime}(t)= \bar h (2- (2-s)t^{1-s})>0 \; \text{and}\; \varrho_0^{\prime\prime}(t)= -\bar h (2-s)(1-s)t^{-s}<0.\]
We choose $\alpha > 1/a$ and set $\varrho(t)= \varrho_0(\alpha t)$, for $t > 0$. Clearly $\varrho^\prime(t)= \alpha \varrho_0^\prime(\alpha t)$ and $\varrho^{\prime\prime}(t)= \alpha^2 \varrho_0^{\prime\prime}(\alpha t)$ for each $0 < t< 1/\alpha$. We claim that
\[\De (\varrho(\delta)) \leq \varrho^{\prime \prime}(\delta)+ R\varrho^\prime (\delta)\delta^{-s}\; \text{in}\; \Om_{1/\alpha}. \]
Let $0\leq \psi \in C_c^\infty(\Om)$ such that $\text{supp}\; \psi \in \Om_{1/\alpha}$. It is already known that $\varrho^\prime(\delta)\psi \geq 0$, $\varrho(\delta) \in H^1_{\text{loc}}(\Om)$, $\varrho^\prime(\delta)\psi \in H^1_0(\Om)$, $\nabla (\varrho(\delta)) = \varrho^\prime(\delta)\nabla \delta$ and $\nabla (\varrho^\prime(\delta)\psi)= \varrho^{\prime\prime}(d)\psi\nabla \delta+ \varrho^\prime(\delta)\nabla \psi$ in the sense of distributions, and $|\nabla \delta|=1$ almost everywhere, by Lemma \ref{dist1}. Since we assumed $\De d \leq Rd^{-s}$ in $\Om_a$, we have
\[- \int_\Om \nabla \delta \nabla (\varrho^\prime(\delta)\psi)~\mathrm{d}x \leq \int_\Om Rd^{-s} (\varrho^\prime(\delta)\psi)~\mathrm{d}x. \]
Therefore, we get
\begin{equation*}
\begin{split}
-\int_\Om \nabla(\varrho(\delta))\nabla \psi \mathrm{d}x = - \int_\Om \varrho^\prime(\delta)\nabla \delta \nabla \psi~\mathrm{d}x&= - \int_\Om \nabla \delta \nabla(\varrho^\prime(d)\psi)~\mathrm{d}x + \int_\Om \varrho^{\prime \prime}(\delta)|\nabla \delta|^2\psi ~\mathrm{d}x\\
& \leq \int_{\Om}  (R\delta^{-s}\varrho^\prime(\delta)+ \varrho^{\prime \prime}(\delta))\psi~\mathrm{d}x.
\end{split}
\end{equation*}
This proves our claim. Using this, now we have
\begin{align*}
-\De (\varrho(\delta))- k_2 (\varrho(\delta))^{-q} & \geq  -\alpha^2\varrho_0^{\prime \prime}(\alpha \delta)- \alpha \varrho_0^{\prime}(\alpha\delta)R\delta^{-s}- k_2 (\varrho_0(\alpha\delta))^{-q}\\
&\geq {(\alpha\delta)}^{-s}(\alpha^2(2-s)(1-s)\bar h- 2\alpha^{1+s}R\bar h -k_2(2\bar h)^{-q})
\end{align*}
in $\Om_{1/\alpha}$. We can fix large $\alpha > 1/a$ such that
\begin{equation}\label{distrel2}
-\De (\varrho(\delta))- k_2 (\varrho(\delta))^{-q} \geq 0 \; \text{ in }\; \Om_{1/\alpha}.
\end{equation}
Next, we will show that $u \leq \varrho(\delta)$. From Lemma \ref{lem2.1}, we get a sequence $\{w_k\} \subset H^1_0(\Om)$ such that  each $w_k$ has compact support in $\Om$, $0\leq w_1\leq w_2\leq \ldots$ and $\{w_k\}$ converges strongly to $u$ in $H^1_0(\Om)$. Suppose $f: \Om \rightarrow [0,1]$ be a $C^\infty$ function such that $f=1$ in $\Om \setminus \Om_{1/\alpha}$ and $f=0$ in $\Om_{1/2\alpha}$. For each $k$, setting $u_k= fu+(1-f)w_k$, we see that $\varrho(\delta) \in H^1_{\text{loc}}(\Om)$, $(u_k-\varrho(\delta))^+ \in H^1_0(\Om)$, and $\text{supp } (u_k-\varrho(\delta))^+ \subset \Om_{1/\alpha}$. Therefore using Lemma \ref{reg}, \eqref{distrel1} and \eqref{distrel2}, we get
\begin{equation*}
\begin{split}
0 &= \int_\Om \nabla u \nabla(u_k-\varrho(\delta))^+~\mathrm{d}x - \la \int_\Om u^{-q} (u_k-\varrho(\delta))^+~\mathrm{d}x\\
& \quad - \int_\Om \int_\Om\frac{|u(y)|^{2^*_\mu}|u(x)|^{2^*_\mu-2}u(x)(u_k-\varrho(\delta))^+(x)}{|x-y|^\mu}~\mathrm{d}y\mathrm{d}x\\
& = \int_\Om \nabla (u-\varrho(\delta)) \nabla(u_k-\varrho(\delta))^+~\mathrm{d}x + k_2  \int_\Om u^{-q} (u_k-\varrho(\delta))^+~\mathrm{d}x - \la \int_\Om u^{-q} (u_k-\varrho(\delta))^+~\mathrm{d}x\\
& \quad - \int_\Om \int_\Om\frac{|u(y)|^{2^*_\mu}|u(x)|^{2^*_\mu-2}u(x)(u_k-\varrho(\delta))^+(x)}{|x-y|^\mu}~\mathrm{d}y\mathrm{d}x +
\int_\Om \nabla \varrho(\delta) \nabla(u_k-\varrho(\delta))^+~\mathrm{d}x \\
& \quad \quad- k_2\int_\Om u^{-q} (u_k-\varrho(\delta))^+~\mathrm{d}x\\
& \geq  \int_\Om \nabla (u-\varrho(\delta)) \nabla(u_k-\varrho(\delta))^+~\mathrm{d}x\\
& = \int_\Om |\nabla(u_k-\varrho(\delta))^+|^2~\mathrm{d}x+ \int_\Om \nabla (u-u_k) \nabla(u_k-\varrho(\delta))^+~\mathrm{d}x.
\end{split}
\end{equation*}
This implies $\|(u_k-\varrho(\delta))^+\| \leq \|u_k-u\| \rightarrow \infty$ as $k \rightarrow \infty$. Therefore, $u \leq \varrho (\delta)$, since $\{(u_k-\varrho(\delta))^+\}$ converges to $(u-\varrho(\delta))^+$ almost everywhere, as $k \rightarrow \infty$. Using $\varrho(\delta) \leq 2\alpha \bar h \delta$, we obtain the conclusion. \QED
\end{proof}

\noi We need the following result (Theorem 3 in \cite{brenir}) to prove our next result.
\begin{Lemma}\label{Hops}
Let $\Om$ be a bounded domain in $\mb R^n$ with smooth boundary $\partial \Om$. Let $u \in L^1_{\text{loc}}(\Om)$ and assume that for some $k \geq 0$, $u$ satisfies, in the sense of distributions
\[
-\De u + ku \geq 0 \; \text{in} \; \Om,\quad
u \geq 0 \;  \text{in}\; \Om.
\]
Then either $u \equiv 0$, or there exists $\e>0$ such that
$u(x) \geq \e \delta(x), \; x \in \Om.$
\end{Lemma}

\begin{Theorem}\label{bdabvdist}
Let $u$ be a positive weak solution of $(P_\la)$, then there exist $L>0$ such that $u \geq L\delta$ in $\Om$.
\end{Theorem}
\begin{proof}
Let $0\leq \psi \in C^\infty_c(\Om)$. Since $u$ is positive weak solution of $(P_\la)$, $u >0$ in $\Om$ and
\begin{align*}
\int_\Om \nabla u \nabla \psi~\mathrm{d}x = \la \int_\Om u^{-q}\psi~\mathrm{d}x+ \int_{\Om}\int_{\Om}\frac{|u(y)|^{2^*_{\mu}}|u(x)|^{2^*_{\mu}-2}u(x)\psi(x)}{|x-y|^{\mu}}~\mathrm{d}x\mathrm{d}y \geq 0.
\end{align*}
Therefore using Lemma \ref{Hops}, we conclude that there must exist a constant $L>0$  such that $u \geq L\delta$ in $\Om$.\QED
\end{proof}
\noi {\bf Proof of Theorem \ref{mainthrm3}:} The proof of Theorem \ref{mainthrm2} follows from Theorem \ref{distrel} and Theorem \ref{bdabvdist}.\\
Using these results, we can say that each positive weak solution of $(P_\la)$ is a classical solution that is $u \in C^\infty(\Om)\cap  C(\bar \Om)$. But actually we can show a little more, see next result.

\begin{Lemma}
Let $q\in (0,\frac{1}{n})$ and  let $u \in H^1_0(\Om)$ be a positive weak solution of $(P_\la)$, then $u \in C^{1+\alpha}(\bar \Om)$ for some $0<\alpha<1$.
\end{Lemma}
\begin{proof}
By previous lemma, we know there exist  a constant $L>0$  such that $u \geq L\delta$ in $\Om$. Since $\frac{1}{\delta}\in L^1(\Om),$ we can find $p>n$ such that $ u^{-q } \in L^{p}(\Om).$
 Also, by Sobolev embedding theorem, we know $u^{2^*_\mu-1} \in L^{t}(\Om)$ where $t=\frac{2^*}{2^*_\mu-1} < n$. Using the Cald$\acute{e}$ron–-Zygmund inequality (refer Theorem B.2 of \cite{struwe}) and since \eqref{chobdd} holds, there exists $M>0$ such that
\begin{equation}\label{C1alpha1}
\begin{split}
|u|_{W^{2,t}(\Om)} &\leq D \left( |u|_t+
\la |u^{-q}|_{t}+ M|u^{2^*_\mu-1}|_{t}\right) \leq D \left( 1+ |\delta^{-q}|_t+ |u^{2^*_\mu-1}|_{t} \right)\\
& \leq D \left( 1+ |u|^{2^*_\mu-1}_{2^*}\right),
\end{split}
\end{equation}
where $D$ is a positive constant which changes at each step. Thus by Sobolev inequality , we have
\begin{equation}\label{C1alpha2}
|u|_{m_1} \leq \left( 1+ |u|^{2^*_\mu-1}_{2^*}\right),
\end{equation}
where $m_1 = \frac{nt}{n-2t}= 2^*\rho_0 $ and $\rho_0 = \frac{n}{(2^*_\mu-1)(n-2t)}>1 $. Thus again using \eqref{C1alpha1} and \eqref{C1alpha2} with replacing $t$ by $m_1$,  we have
\[|u|_{m_2} \leq \left( 1+ |u|^{2^*_\mu-1}_{m_1}\right) \leq \left( 1+ |u|^{(2^*_\mu-1)^2}_{2^*}\right) \]
where $m_2 = \frac{nm_1}{(2^*_\mu-1)(n-2m_1)}= 2^* (\rho_0)^2 $. Repeating the same process and replacing $m_2$ by $m_1$, we can get a positive integer $m$ such that $\frac{2^*\rho_0^m}{2^*_\mu-1}>n $ and therefore using Sobolev embedding we get
\begin{align*}
|u|_{C^{1+\alpha}(\bar \Om)} & D\leq |u|_{W^{2, \min\{\rho,\frac{2^*\rho_0^m}{2^*_\mu-1}\}}} \leq D \left(|u^{-q}|_{\rho}+ |u|_{\frac{2^*\rho_0^m}{2^*_\mu-1}} \right)\\
& \leq D \left(1+ |u|^{(2^*_\mu-1)^m}_{2^*} \right) \leq D \left(1+ \|u\|^{(2^*_\mu-1)^m} \right),
\end{align*}
for $\alpha \in (0,1)$ and $D$ being a positive constant which changes at each step. This completes the proof. \QED
\end{proof}

\end{document}